\documentclass[11pt]{article}
\usepackage{amssymb}	
\usepackage{amsmath}    
\usepackage{epsfig}
\usepackage{multirow}
\usepackage{longtable}
\usepackage{algorithmic}
\usepackage[margin=1in]{geometry}
\usepackage{setspace}

\newenvironment{proof}[1][Proof:]{\begin{trivlist} 
\item[\hskip \labelsep {\bfseries #1}]}{\end{trivlist}} 
\newcommand{\qed}{\nobreak \ifvmode \relax \else \ifdim\lastskip<1.5em \hskip-\lastskip \hskip1.5em plus0em minus0.5em \fi \nobreak \vrule height0.75em width0.5em depth0.25em\fi} 
\def\0{\bf \0}
\def\A{{\bf A}}
\def\B{{\bf B}}

\def\E{{\bf E}}
\def\F{{\bf F}}

\def\I{{\bf I}}

\def\M{{\bf M}}

\def\0{{\bf 0}}

\def\R{\mathbb{R}}

\def\T{{\bf T}}

\def\a{{\bf a}}
\def\b{{\bf b}}
\def\c{{\bf c}}

\def\e{{\bf e}}
\def\f{{\bf f}}

\def\u{{\bf u}}
\def\v{{\bf v}}
\def\w{{\bf w}}
\def\x{{\bf x}}
\def\y{{\bf y}}
\def\z{{\bf z}}

\def\Tr{{\rm T}}
\def\T{{\rm T}}

\newtheorem{algorithm}{Algorithm}[section]

\newtheorem{theorem}{Theorem}[section]

\newtheorem{corollary}{Corollary}[section]
\newtheorem{remark}{Remark}[section]

\usepackage{color}
\usepackage[normalem]{ulem}

\newcommand*\Bell{\ensuremath{\boldsymbol\ell}}

\begin{document}
\title{On the facet pivot simplex method for linear programming}
\author{
Yaguang Yang\thanks{Goddard Space Flight Center, NASA, 
8800 Greenbelt Rd, Greenbelt, MD 20771.
Email: yaguang.yang@verizon.net.} 
}


\maketitle    
\begin{abstract}

Dantzig's {\it vertex pivot} simplex method 
has been published for more than seven decades. Amazingly, it remains
one of the most efficient methods to solve linear programming (LP)
problem after numerous efforts trying to find some better
methods. In this paper, we propose a {\it facet pivot} simplex method
and demonstrate by numerical testing that the new method is 
very promising compared to the vertex pivot method. Since there is no polynomial pivot simplex algorithm for linear programming problems after many decades of effort, we hope that this new type of pivot algorithm will give us some hope to find a polynomial pivot simplex method for linear programming problems. A Matlab implementation of the facet pivot algorithm and Netlib benchmark test problems are available in Matlab file exchange website.

\end{abstract}

{\bf Keywords:} Facet pivot algorithm, simplex method, 
linear programming, Klee-Minty cube, cycling problem.

 
\section{Introduction}

Since Dantzig invented the vertex simplex method in the 1940s 
\cite{dantzig49}, the linear programming (LP) problem has been one
of the mostly studied mathematical problems. Although researchers
have engaged tremendous efforts searching for more efficient algorithms,
the vertex simplex method still stands as one of the most efficient
methods (especially for small and middle-size problems). Researchers have been very curious about what its 
performance is in the theoretically best scenario and in the practically 
worst scenario. Since the vertex simplex method searches for the optimizer among
vertices of a polytope along a path of edges, the theoretically best performance 
scenario for the vertex simplex method is related to the length of the diameter 
of the polytope. The famous Hirsch conjecture \cite{dantzig63} 
states that the shortest path length is bounded by $n-d$ 
(where $n$ is the number of facets and $d$ is the dimension 
of the polytope) which was disproven by Santos \cite{santos12}.
The best-known upper bounds for the polytope diameter 
\cite{kalai92,todd14,sukegawa19} are sub-exponential, which means that
vertex pivot algorithms may be very expensive based on the current estimate.
However, experts believe that the true upper bound of the diameter
(the best scenario) should be a polynomial of $n$ and $d$ \cite{santos12a}.

Although the theoretically best scenario looks bright, the research on the
practically worst scenario seems not, as Klee and Minty \cite{km72} 
discovered that the number of iterations used to find an optimal solution
in the worst case for Dantzig's {\it vertex pivot} simplex method increases 
exponentially as a function of the problem size $d$, and subsequently,
researchers showed that many popular vertex simplex algorithms 
require an exponential number of steps to find an optimizer
\cite{az96,ac78,af17,friedmann11,fhz11,gs79,hz15,jeroslow73,ziegler04}. 
Searching for efficient polynomial algorithms motivated the ellipsoidal 
method \cite{khachiyan} and the interior point method \cite{karmarkar84}, 
the latter has become very popular since the 1980s \cite{yang20}. 
On the other hand, there is still much research on novel vertex pivot simplex 
methods \cite{tz93}, and more recently, the exterior point simplex 
method \cite{paparrizos93} which relaxes the feasibility requirement 
while searching for the optimizer, and the double pivot simplex method 
\cite{ve18,yang20a} which updates two vertices at a time. 

Almost all simplex algorithms use {\it vertex pivot}, 
including dual simplex algorithms 
\cite{koberstein05,kostina02,lemke54,kgk79,maros03,pan05,vanderbei14}.
The vertex pivot simplex algorithms consider the standard linear
program problem with the constraints of $\A^I \x =\b^I$ and 
$\x \geq \0$, and the number of the columns of $\A^I$ is 
greater than the number of rows of $\A^I$. These algorithms utilize 
vertex geometry, and the base is formed from a subset of
primary variables (a primary vertex) or a subset of dual variables
and dual slack variables (a dual vertex); 
every iterate is a primal (or dual) feasible solution; 
and the iterate moves from a vertex to an adjacent vertex. Very recently, 
Liu et. al. \cite{ltz21} published their brilliant idea, which is based 
on the facet geometry. The facet pivot method considers the
canonical linear program problem with the constraints of $\A^J \x \ge \b^J$
and ${\Bell} \leq \x \leq \u$, where the number of the rows of constraints is greater
than the number of the columns of constraints\footnote{We can rewrite the two sets of constraints $\A^J \x \ge \b^J$
and ${\Bell} \leq \x \leq \u$ as one $\A \x = \left[ \begin{array}{c} \A^J \\ \I \\ -\I  \end{array} \right] \x \geq \left[ \begin{array}{c} \b^J \\ \Bell \\ -\u \end{array} \right]$.}. Every iterate is a basic
solution but not necessarily a feasible solution until an optimal 
solution is found.

The facet pivot simplex method deals with a format of linear 
programming problem that is similar to, but more general than, 
the dual simplex method. However, the two methods are 
fundamentally different in their basic ideas. Dual simplex method was
first introduced by Lemke \cite{lemke54} in 1954. It became a 
computationally competitive alternative of the primary simplex method 
after Forrest and Goldfarb \cite{fg92} published their steepest edge 
simplex algorithm. Koberstein and others (see \cite{koberstein08,ks07}
and references therein) carefully implemented the algorithm. 
The iterates of the dual simplex method are {\it dual basic feasible} 
in every iteration and the algorithm checks the feasibility of the 
{\it primal feasibility} to determine if an optimal solution is found. 
If it is not, it determines a leaving vector in the {\it dual base} 
(using the information of the primal variables) and 
then finds an entering vector corresponding to a nonbase 
dual variable so that the new basic solution is feasible 
for the {\it dual problem}. The search moves from one vertex to the 
next vertex of the {\it dual} problem until a basic solution is 
feasible for the primary problem. The computation involves both 
{\it primal} and {\it dual} variables in both primal and dual problems 
\cite{koberstein05}. On the other hand, the facet pivot method 
considers only the primal problem with equality, inequality and 
boundary constraints, the iterate is a basic but not necessarily a 
feasible solution of the primal problem until an optimal solution is 
found, and all steps involve only the primal but not the dual problem. 
Note that the optimal solution of an LP problem can always
be found in a vertex and a vertex is always formed by a set
of $d$ independent facets ($d$ is the dimension of the 
polytope). The facet pivot algorithm starts from an initial facet base 
which is composed of $d$ facets (the constraints), it then finds
an entering facet and a leaving facet to form a new base.
Therefore, every iterate improves feasibility by replacing one 
facet (constraint) with a different facet (constraint), 
and all iterates are facet basic (to be defined later) but not 
feasible until an optimal solution is found. For canonical form 
of the LP problem, assuming one knows the optimal solution 
and starts from the most distant vertex, it needs at least
$n-d$ iterations for vertex simplex method to reach the optimal 
solution \cite{santos12} which is equal to the diameter of the 
convex polytope. As the number of inequality constraints $n$
goes to infinite, the iteration number goes to infinite (this is true
for both primal and dual vertex simplex methods). But if one 
knows the optimal solution which is a vertex of the intersection
of $d$ facets, it needs at most $d$ iterations for facet pivot method
to reach the optimal solution no matter how many inequality 
constraints are.

Another obvious computational advantage of using the general 
formulation considered by the facet pivot method is that 
one does not need to convert inequality constraints into equality 
constraints which avoids introducing unnecessary 
slack variables. Using this formulation, the problem size is
significantly smaller than the problem size of the standard LP problem,
especially when the original problem has many inequality constraints. 
It is also more natural to handle lower and upper bound constraints 
and free variables. Moreover, the facet pivot method with the general
form does no need a possibly very expensive Phase I process (the 
Phase I process of dual vertex simplex method may be cheap for
some problems but not always) to find a starting point.

Notice that, by duality, the facets of the primal problem
are the vertices of the dual problem, the facet pivot method may have
some connections to the dual vertex pivot method, but the former involves 
a smaller problem (primal problem is smaller than primal plus dual problem).
This means that the facet pivot method is conceptually much simpler 
because it does not rely on duality. In section \ref{testSection},
we will demonstrate by numerical test that the facet pivot method
is more attractive than the dual vertex pivot method, at least for the
tested Netlib problems. Although the facet pivot simplex 
method is not a dual simplex method, to find some possible relationship may 
be an interesting future research topic beyond the scope of this paper.

This paper adopts the ideas of \cite{ltz21}, improves
the results, and devises several algorithms that are suitable 
for computer programs rather than tabular procedures.  
In addition, we provide a series of technical results to 
show that these algorithms are well defined, and they find 
the optimal solution in finite iterations. We implement the Dantzig's simplex algorithm, two
facet simplex algorithms, a dual simplex algorithm, and two interior-point algorithms
as Matlab functions and report test results for various benchmark
problems, including a small set of known cycling LP problems
\cite{yang21}, three variants of Klee-Minty problems 
\cite{greenberg97,km11,ibrahima13} with different sizes, 
randomly generated standard and general LP problems with different distributions, 
and the Netlib benchmark problems \cite{bdgr95}. These preliminary
tests demonstrate the efficiency and effectiveness of the 
facet pivot simplex algorithm. It shows that the facet pivot simplex 
method is very promising compared to Dantzig's simplex and
dual simplex algorithms for general LP problems.

In the remainder of this paper, small letters with bold 
font represent vectors, capital letters with bold font represent 
matrices, and letters with normal font represent scalars. 
For a vector $\x$, we denote by $x_i$ its $i$-th component. 
Also, we use superscript $k$ to denote the iteration count.
Therefore, the scalar, vector, set, and matrix at the $k$-th 
iteration are denoted, for example, as $x_i^k$, $\x^k$, $B^k$,
and $\A_{B^k}$. To save space, we write the stacked column vector 
$\x =[\x_1^{\Tr}, \x_2^{\Tr}]^{\Tr}$ as $\x=(\x_1, \x_2)$,
and stacked matrices $\A=\left[ \A_1^{\Tr}, \A_2^{\Tr} \right]^{\Tr}$
as $\A= \left( \A_1,\A_2 \right)$.
The remainder of the paper is organized as follows.
Section 2 describes the standard general form of the linear 
programming problems. Section 3 presents some mathematical results
that will be used to justify the development of the facet pivot 
simplex algorithm. Section 4 provides the detailed steps of the facet
pivot simplex algorithm and the reasoning behind every step, 
including proof of the claim that the algorithm finds 
the optimal solution in finite steps. 
Section 5 discusses some important implementation
details for readers who are interested in repeating the reported
results given in this paper. Section 6 concludes the paper
with some remarks.

\section{The standard general form of the LP problem}
\label{proposedAlgo}

We consider the general linear programming problem presented
as follows:
\begin{subequations}
\begin{align}
\min \hspace{0.5in}  &  \c^{\Tr}\x, \label{1a} \\ 
\mbox{\rm subject to} \hspace{0.3in}
&  \A^I \x = \b^I,  \label{1b} \\
&  \A^J\x \ge \b^J,  \label{1c} \\
& \u \ge \x \ge {\Bell},   \label{1d} 
\end{align}
\label{generalLP}
\end{subequations}
where $\A^I \in {\R}^{m \times d}$, $\A^J \in {\R}^{n \times d}$, 
$\b^I  \in {\R}^{m}$, $\b^J  \in {\R}^{n}$, $\Bell \in {\R}^{d}$,
$\u \in {\R}^{d}$, and $\c \in {\R}^{d}$ are given matrices and
vectors respectively. Vector $\x \in {\R}^d$ is composed of variables 
to be optimized. We say that an $\x$ is a {\tt feasible} solution of
LP if $\x$ satisfies all the constraints of ({\ref{generalLP}). We refer 
(\ref{generalLP}) as the general form of LP because the standard 
form of the LP problem
\begin{eqnarray}
\begin{array}{cl}
\min  &  \c^{\Tr}\x, \\ 
\mbox{\rm subject to} 
&  \A^I\x=\b^I, \\
& \x \ge \0 ,  
\end{array}
\label{standardLP}
\end{eqnarray}
and the canonical form of the LP problem
\begin{eqnarray}
\begin{array}{cl}
\min  &  \c^{\Tr}\x, \\ 
\mbox{\rm subject to} 
&  \A^J \x \ge \b^J, \\
& \x \ge \0 ,  
\end{array}
\label{canonicalLP}
\end{eqnarray}
are special cases of the general form of the LP. Clearly, all LP 
problems can be written as either a standard form which is useful 
for a unified theoretical analysis of the vertex pivot simplex method, 
or a canonical form which is a better description from a geometric 
point of view (convex polytope) for the LP problem (every row 
of the constraint forms a facet of the convex polytope). But using 
the general form of the LP, as we will show, 
normally leads to a more efficient and 
convenient (without introducing many relaxation variables and
performing extra conversion work) algorithm for the 
general LP problems which are frequently met in real applications,
for example, the facet enumeration problem \cite{yang19}.
We emphasize that although the general  
form of LP can be converted to either standard LP or canonical LP,
it is better solved by directly applying the facet pivot simplex method rather than other simplex methods that require conversions.

\begin{remark}
Let $x_i$, $i=1, \ldots, d$, be the $i$-th variable 
of $\x$. For a free or a constrained on one side variable $x_i$,
we can always represent it as the general form. For example, 
if $x_i$ is a free variable, we can add two trivial constraints
as $-M \le x_i \le M$, where $M$ is a big positive constant; 
if $x_i \ge 0$, we can add a trivial constraint so that it meets 
$M \ge x_i \ge 0$; if $x_i \le 0$, we can add a trivial constraint 
so that it meets $-M \le x_i \le 0$. This conversion seems to increase 
the complexity of the problem, but as it will be seen, it provides a convenient
initial basic solution which avoids the expensive Phase 1 computation
in the traditional simplex method.
\label{1stRemark}
\end{remark}

From Remark \ref{1stRemark}, one can always write a
given LP problem as the general form of (\ref{generalLP}). 
Let $c_k$, $k=1, \ldots,d$, be the $k$-th component of $\c$; 
$\a_i$, $i=1, \ldots, m$, be the $i$-th row (facet) of $\A^I$;
$\a_j$, $j=m+1, \ldots, m+n$, be the $j$-th row (facet) 
of $\A^J$. Similar notations are used for $b_i$ with 
$i=1, \ldots, m$; $b_j$ with $j=m+1, \ldots, m+n$; 
$\ell_k$ with $k=m+n+1, \ldots, m+n+d$;
and $u_k$ with $k=m+n+d+1, \ldots, m+n+2d$. 
Noticing that $\u \ge \x \iff -\x \ge -\u$ (where $\iff$ means
equivalent), we denote
\begin{subequations}
\begin{align}
\sigma_i (\x) =\a_i\x-b_i, & \hspace{0.1in} i=1, \ldots, m, \\
\sigma_j (\x) =\a_j\x- b_j, &\hspace{0.1in} j=m+1, \ldots, m+n, \\
\underline{\sigma}_k (\x) = x_k- \ell_k, & 
\hspace{0.1in} k=m+n+1, \ldots, m+n+d, \\
\overline{\sigma}_k (\x) =- x_k+ u_k, & \hspace{0.1in} 
k=m+n+d+1, \ldots, =m+n+2d, 
\end{align}
\label{sigma}
\end{subequations}
which will be used to measure the constraint violations.
To keep the notation simple, we often omit $\x$ and simply use 
$\sigma_i$, $\sigma_j$, $\underline{\sigma}_k$, and $\overline{\sigma}_k$,
but remember that they are functions of $\x$.


Let $\e_i^{\T}$ be the $i$-th row of the $d$-dimensional identity 
matrix $\I$. The {\tt standard general} form LP is given as follows:
\begin{subequations}
\begin{align}
\min  \hspace{0.5in}  &  \sum_{i=1}^{d}
\bar{c}_i \bar{\e}_i^{\Tr}\x, \hspace{0.1in} \bar{c}_i \ge \0 
 \label{5a}  \\ 
\mbox{\rm subject to}  \hspace{0.3in} 
&  \A^I \x=\b^I, \label{5b}  \\
&  \A^J\x \ge \b^J, \label{5c}   \\
&  \E \x \ge {\b_L}, \label{5d}   \\
&  \F \x \ge {\b_U}, \label{5e}  
\end{align}
\label{stdGeneralLP}
\end{subequations}
where (\ref{5b}) and (\ref{5c}) are the same as  (\ref{1b}) 
and (\ref{1c}); some conversions from (\ref{generalLP}) to
(\ref{stdGeneralLP}) will be needed for (\ref{1a})
and (\ref{1d}), which are described as follows.
The $i$-th row (or the facet) of 
$\E$ is denoted as $\bar{\e}_i^{\Tr}$ which is 
either $\e_i^{\T}$ or $-\e_i^{\T}$ depending on the sign of $c_i$: 
(i) if $c_i \ge 0$, then, $\bar{c}_i=c_i$, $\bar{\e}_i^{\Tr}=\e_i^{\Tr}$,
we set $b_{L_i} \leftarrow \ell_i$ in (\ref{5d}) to represent 
$x_i \ge \ell_i$; the $i$-th row (facet) of $\F$ is given as 
${\f}_i^{\Tr}=-\e_i^{\Tr}$ and set $b_{U_i} \leftarrow (-u_i)$ to 
represent the constraint $-x_i \ge -{u_i}$ in (\ref{5e}); 
and (ii) if $-c_i > 0$, then, this item in the objective function can be 
written as $c_i x_i=(-c_i) (-\e_i^{\T})\x=\bar{c}_i \bar{\e}_i^{\T}\x$, 
we set $\bar{c}_i \leftarrow (-c_i)$ in the objective function
(\ref{5a}), and  rewrite the corresponding inequalities in (\ref{1d}) 
as $-u_i \le -x_i \le -\ell_i$ or 
$-u_{i} \le -\e_{i}^{\T}\x \le  -\ell_{i}$, therefore, we set
$\b_{L_i} \leftarrow (-u_i)$ to represent $-\e_{i}^{\T}\x \ge -u_i$ 
in (\ref{5d}), and $\f_i^{\T}=\e_i^{\T}$ and $b_{U_i}  \leftarrow \ell_i$ 
to represent $x_i \ge \ell_i$ in (\ref{5e}). This completes 
the process of converting the general form LP to the standard 
general form LP of (\ref{stdGeneralLP}). To keep the notation simple, 
we will use $c_i$ for $\bar{c}_i$ in the remainder 
of the paper, and we discuss the standard general LP. Let 
\begin{equation}
\A=\left[ \begin{array}{c}  \A^I \\ \A^J \\ \E \\ \F \end{array} \right],
\hspace{0.2in} \b =\left[ \begin{array}{c}  \b^I \\ \b^J \\ \b_L \\ \b_U
\end{array} \right],
\label{Amatrix}
\end{equation}
where $\E \in \R^{d \times d}$ and $\F \in \R^{d \times d}$
are diagonal and full rank matrices. 
We denote by $\A_B$, a sub-matrix of $\A$, which is composed of $d$ 
linear independent row vectors of $\A$, as a {\tt base} of $\A$.
The rows (facets) of $\A_B$ are named as basic facets.
We denote by $\A_N$, which is composed of the remaining $m+n+d$ row
(facet) vectors of $\A$, as a {\tt non-base} sub-matrix of $\A$.
The rows (facets) of $\A_N$ are named as non-basic rows (facets). Similarly,
we denote by $\b_B$, a $d$-dimensional sub-vector of $\b$, 
corresponding to the rows of $\A_B$; and by $\b_N$ a $(m+n+d)$
-dimensional sub-vector of $\b$, corresponding to the rows of $\A_N$. 
We say that a vector $\x$ is a {\tt basic} solution of ({\ref{stdGeneralLP})
if it satisfies 
\begin{equation}
\A_B\x=\b_B. 
\label{basicS}
\end{equation}
Similarly, we say that a vector $\x$ is a 
{\tt basic feasible} solution of Problem ({\ref{stdGeneralLP}) if $\x$ 
is both a basic and a feasible solution. Another good feature of the 
standard general form is that we may choose the rows (facets) of $\E$ 
as the initial base, i.e., $\A_{B^0}=\E$. It is worthwhile to note 
that $\A_{B^0}=\E$ corresponds to the inequality constraints, and
$c_i$, $i=1,\ldots,d$, are the coefficients such that 
$\c=\sum_{i=1,\ldots,d} c_i \bar{\e}_i$, moreover, $c_i^0=c_i \ge 0$.
In addition, the initial basic solution $\x^0$ can easily be obtained by solving 
$\E \x^0 = {\b_L}$ as $\E$ is a diagonal matrix and all diagonal
elements are either $1$ or $-1$. 

A possibly very expensive step in the traditional vertex pivot simplex method 
is the so-called Phase-I step, which is aimed at finding a feasible initial 
point. By rewriting the general LP problem to the {\tt standard general} 
form LP, the facet pivot simplex method does not need this 
additional step. The key ideas are to (a) represent objective vector 
$\c$ as the linear combination of the basic facets $\A_{B^k}$, i.e.,
$\c^{\T}=\sum_{i \in B^k} y_{ci}^k \a_i$ at the $k$-th iteration, 
and (b) keep $y_{ci}^k \ge 0$ ($y_{ci}^k$ corresponds to the $i$-th 
facet of the base at the $k$-th iteration) 
for all $i$ and $k$ in the base of the $k$-th iteration. 
We will see that keeping $y_{ci}^k \ge 0$ for all inequality constraints 
in the base for all iterations $k$ is very important for the use of
Farkas Lemma to justify the facet pivot simplex algorithm.


Without the loss of the generality, we make the following assumptions
throughout this paper.

\vspace{0.1in}
{\bf Assumptions}
\begin{itemize}
\item[1.] $rank(\A^I) < d$.
\item[2.] $rank(\A) = d$.
\end{itemize}
The first assumption means that the optimization problem is not trivial,
otherwise the feasible solution is either unique or does not exist. 
The second assumption is not necessary but makes the discussion simple.

\section{Optimality theorems for the facet pivot simplex method}

The facet pivot simplex method is based on Farkas Lemma which is presented
as follows (see reference \cite[Lemma 1.1]{yang20}).

\begin{theorem}[Farkas Lemma]
Let $\c \in \R^{d}$, $\x \in \R^{d}$, $\y \in \R^{n}$, and
$\A \in \R^{n \times d}$. Then, exactly one of the following 
systems holds.
\begin{itemize}
\item[(i)] $\A \x \ge \0$ and $\c^{\Tr} \x <0$.
\item[(ii)] $\A^{\Tr} \y =\c$ and $\y \ge \0$.
\end{itemize}
\label{farkas}
\end{theorem}

Let $\mathcal{I}$ be the index set of all (including lower and
upper boundary) inequality constraints in (\ref{stdGeneralLP}), 
$\mathcal{E}$ be the index set of all equality 
constraints in (\ref{stdGeneralLP}), $\mathcal{B}$ be the
index set of the base $B$, and $\mathcal{N}$ be the
index set of the non-base $N$. We will use $\A_{I}$, $\A_{E}$, 
$\A_{B}$, and $\A_{N}$ to denote sub-matrices corresponding 
to the index sets $\mathcal{I}$, $\mathcal{E}$, $\mathcal{B}$, 
and $\mathcal{N}$ respectively. We will use the same partitions 
for $\b_{I}$, $\b_{E}$, $\b_{B}$, and $\b_{N}$. Therefore,
\[
\A_{I}=\left[ \begin{array}{c}  \A^J \\ \E \\ \F \end{array} \right],
\hspace{0.2in} \b_{I} =\left[ \begin{array}{c}  \b^J \\ \b_L \\ \b_U
\end{array} \right],
\]
$\A_{E}=\A^I$, $\b_{E}=\b^I$, and the {\tt standard general} form LP
can be written as follows:
\begin{subequations}
\begin{align}
\min  \hspace{0.5in}  &  \sum_{i=1}^{d}
{c}_i \bar{\e}_i^{\Tr}\x, \hspace{0.1in} c_i \geq \0 
 \label{6a}  \\ 
\mbox{\rm subject to}  \hspace{0.3in} 
&  \A_E \x=\b_E, \label{6b}  \\
&  \A_I \x \geq \b_I. \label{6c}   
\end{align}
\label{stdGeneralLP1}
\end{subequations}

Denote by $\mathcal{I}_0$ and $\mathcal{E}_0$ as the index 
subsets of $\mathcal{I}$ and $\mathcal{E}$ respectively such that 
$\mathcal{I}_0 \cup \mathcal{E}_0$ forms the index set of 
the base $B$, i.e., $\mathcal{B}=\mathcal{I}_0 \cup \mathcal{E}_0$;
by $\mathcal{I}_1 =\mathcal{I} \setminus \mathcal{I}_0$
the index subset of inequality constraints not in the base $B$,
and by $\mathcal{E}_1 =\mathcal{E} \setminus \mathcal{E}_0$
the index subset of equality constraints not in the base $B$.
Clearly, we have $\mathcal{N}=\mathcal{I}_1 \cup \mathcal{E}_1$.
Finally, we denote by $\A_{I_0}$, $\A_{E_0}$, $\A_{I_1}$, 
$\A_{E_1}$ the sub-matrices corresponding to $\mathcal{I}_0$,
$\mathcal{E}_0$, $\mathcal{I}_1$, and $\mathcal{E}_1$; 
and by $\b_{I_0}$, $\b_{E_0}$, $\b_{I_1}$, and $\b_{E_1}$
for the same partitions of $\b$.

The following theorem (see \cite[Corollary 1]{ltz21}), which can 
easily be derived from Farkas Lemma, is useful in the development 
of the facet pivot simplex algorithm. Denote 
$\A_B=\left[ \begin{array}{c} \A_{E_0} \\ 
\A_{I_0} \end{array} \right] \in \R^{d \times d}$. 
Let vectors $\y=\left[ \begin{array}{c} 
\y_{E_0} \\ \y_{I_0} \end{array} \right]  \in \R^{d}$ 
and $\b_B=\left[ \begin{array}{c} \b_{E_0} \\ \b_{I_0} 
\end{array} \right]  \in \R^{d}$ be partitioned such that
the row (facet) indices of $\A_{E_0}$, $\y_{E_0}$, and
$\b_{E_0}$ are identical, and the row (facet) indices of 
$\A_{I_0}$, $\y_{I_0}$, and $\b_{I_0}$ are identical.

\begin{theorem}[\cite{ltz21}]
Let $\x \in \R^{d}$ be a basic solution of (\ref{stdGeneralLP1}), 
i.e., $\A_B \x =\b_B$. Let $\bar{\x} \in \R^{d}$ be a 
feasible solution of (\ref{stdGeneralLP1}),
i.e., $\A_{E} \bar{\x}  = \b_{E}$ and $\A_{I} \bar{\x}  \geq \b_{I}$. 
Then, exactly one of the following systems holds but not both.
\begin{itemize}
\item[(i)] $\A_{I_0} \x \geq \b_{I_0}$, $\A_{E_0} \x = \b_{E_0}$,
and $\a_{r} \x <\a_r \bar{\x}$.
\item[(ii)] $\A_{E_0}^{\T}\y_{E_0}-\A_{I_0}^{\T}\y_{I_0} =\a_r^{\T}$ 
and $\y_{I_0} \geq \0$.
\end{itemize}
\label{farkasD}
\end{theorem}
\begin{proof}
Since $\A_{E_0} \x = \b_{E_0}$ is equivalent to 
$\b_{E_0} \geq \A_{E_0} \x \geq \b_{E_0}$, which can be written as
$\A_{E_0} \x \geq \b_{E_0}$ and $-\A_{E_0} \x \geq -\b_{E_0}$.
Using the fact that $\A_{E_0} \bar{\x}  = \b_{E_0}$, we have
$\A_{E_0} (\x- \bar{\x} ) \geq \0$ and 
$-\A_{E_0} (\x- \bar{\x} ) \geq \0$. Combining
$\A_{I_0}{\x}  = \b_{I_0}$
and $\A_{I_0} \bar{\x}  \geq \b_{I_0}$, we have
$-\A_{I_0} (\x- \bar{\x} ) \geq \0$. Therefore, system (i) is equivalent to
\begin{subequations}
\begin{align}
\A_{E_0} (\x- \bar{\x} ) \geq \0,  \\
-\A_{E_0} (\x- \bar{\x} ) \geq \0, \\
-\A_{I_0} (\x- \bar{\x} ) \geq \0, \\
\a_{r} (\x- \bar{\x} )   < 0.
\end{align}
\label{equivI}
\end{subequations} 
Let $\A^{\T}=[ \A_{E_0}^{\T},~ -\A_{E_0}^{\T},~ -\A_{I_0}^{\T}]$
and $\c^{\T}=\a_{r}$, then, system (i) of this theorem is equivalent
to the relations of (\ref{equivI}), which is the same as system (i) of the Farkas 
Lemma. System (ii) of the Farkas Lemma ($\A^{\Tr} \y =\c$ and 
$\y \ge \0$)  is equivalent to 
\begin{eqnarray}
\A^{\T} \y &=&
[ \A_{E_0}^{\T},~ -\A_{E_0}^{\T},~ -\A_{I_0}^{\T}]
\left[ \begin{array}{c} \y_1 \\ \y_2 \\ \y_{I_0} \end{array} \right]
\nonumber \\
&=&  \A_{E_0}^{\T} (\y_1-\y_2) - \A_{I_0}^{\T} \y_{I_0}
: = \A_{E_0}^{\T}  \y_{E_0} - \A_{I_0}^{\T} \y_{I_0} = \a_r^{\T}
\end{eqnarray}
with $\y_1-\y_2=\y_{E_0}$ and $ \y_{I_0} \ge \0$, which is system 
(ii) of this theorem. 
\hfill \qed
\end{proof}

\begin{corollary}
Let $\x \in \R^{d}$ be a basic solution of (\ref{stdGeneralLP1}), 
i.e., $\A_B \x =\b_B$. Let $\bar{\x} \in \R^{d}$ be a 
feasible solution of (\ref{stdGeneralLP1}),
i.e., $\A_{E} \bar{\x}  = \b_{E}$ and $\A_{I} \bar{\x}  \ge \b_{I}$. 
Then, exactly one of the following systems holds but not both.
\begin{itemize}
\item[(i)] $\A_{I_0} \x \ge \b_{I_0}$, $\A_{E_0} \x = \b_{E_0}$,
and $\a_{r} \x >\a_r \bar{\x}$.
\item[(ii)] $\A_{E_0}^{\T}\y_{E_0}+\A_{I_0}^{\T}\y_{I_0} =\a_r^{\T}$ 
and $\y_{I_0} \ge \0$.
\end{itemize}
\label{farkasC}
\end{corollary}
\begin{proof}
Since $\A_{E_0} \x = \b_{E_0}$ is equivalent to 
$\b_{E_0} \ge \A_{E_0} \x \ge \b_{E_0}$, which can be written as
$\A_{E_0} \x \ge \b_{E_0}$ and $-\A_{E_0} \x \ge -\b_{E_0}$.
Using the fact that $\A_{E_0} \bar{\x}  = \b_{E_0}$, we have
$\A_{E_0} (\x- \bar{\x} ) \ge \0$ and 
$-\A_{E_0} (\x- \bar{\x} ) \ge \0$, which are equivalent to
$-\A_{E_0}(\bar{\x} -\x ) \ge \0$ and
$\A_{E_0} (\bar{\x} -\x) \ge \0$. Combining
$\A_{I_0}{\x}  = \b_{I_0}$
and $\A_{I_0} \bar{\x}  \ge \b_{I_0}$, we have
$-\A_{I_0} (\x- \bar{\x} ) \ge \0$, which is equivalent to
$\A_{I_0} (\bar{\x} -\x) \ge \0$. Therefore, system (i) of
this Corollary is equivalent to
\begin{subequations}
\begin{align}
\A_{E_0} (\bar{\x} -\x) \ge \0,  \\
-\A_{E_0}(\bar{\x} -\x ) \ge \0, \\
\A_{I_0} (\bar{\x} -\x) \ge \0, \\
\a_{r} (\bar{\x} -\x )   < 0.
\end{align}
\label{equivII}
\end{subequations} 
Let $\A^{\T}=[ \A_{E_0}^{\T},~ -\A_{E_0}^{\T},~\A_{I_0}^{\T}]$
and $\c^{\T}=\a_{r}$, then, following the same argument as we
have done in Theorem \ref{farkasD} proves the corollary.
\hfill \qed
\end{proof}

The optimality condition of (\ref{stdGeneralLP1}) can be derived  
from Corollary \ref{farkasC}. This condition is presented as the
following theorem.

\begin{theorem}[optimality condition]
Assume that the base matrix is partitioned as $\A_B=(\A_{E_0},\A_{I_0})$ and
$\x^*$ is a basic feasible solution of (\ref{stdGeneralLP1}), i.e., $\A \x^* \ge \b$
and $\A_B \x^* = \b_B$. The latter implies $\A_{I_0} \x^* = \b_{I_0}$,
where the index set $I_0$ is a subset of $\mathcal{I}$, and the 
corresponding inequality constraints are active. If
\begin{equation}
\c=\A_{E_0}^{\T} \y_{E_c} + \A_{I_0}^{\T} \y_{I_c},
\hspace{0.1in} \y_{I_c} \geq \0,
\label{optimal}
\end{equation}
then, $\x^*$ is an optimal solution of (\ref{stdGeneralLP1}).
\label{optimality}
\end{theorem}
\begin{proof}
Assume system (\ref{optimal})
is true, since it is identical to the system (ii) of Corollary \ref{farkasC},
this means that the system (i) of Corollary \ref{farkasC} is false.
Since $x^*$ is a basic solution, it follows that 
$\A_{I_0} \x^* = \b_{I_0}$ and $\A_{E_0} \x^* = \b_{E_0}$ hold.
Therefore, for any feasible solution $\bar{\x}$ of (\ref{stdGeneralLP1}),
$\c \x^* > \c \bar{\x}$ must be incorrect, i.e., $\c \x^* \le \c \bar{\x}$ 
must hold. Since $\x^*$ is feasible, it is an optimal solution of
(\ref{stdGeneralLP1}).
\hfill \qed
\end{proof}

Clearly, for the standard general form of LP (\ref{stdGeneralLP}), 
there is an initial point $\x^0$ satisfying $\E \x^0 =\b_L$ (because
$\E \x^0 =\A_{B^0} \x^0 = \b_{B^0}=\b_L$) and equation 
(\ref{optimal}) holds (because $\c = \E^{\T} \bar{\c}=\A_{I_0}^{\T} \y_{I_c}$ 
and $\bar{\c} =\y_{I_c} \ge \0$). 
Assume that $\x^0$ is not feasible, otherwise, an optimal solution is 
found according to Theorem \ref{optimality}. In general, at the $k$-th 
iteration, the facet pivot simplex method generates an $\x^k$ and a 
$\y_c^k=(\y_{E_c^k},\y_{I_c^k})$ that satisfy 
\[
\A_{B^k} \x^k = \b_{B^k}, \hspace{0.1in} 
\c=\A_{E_0^k}^{\T} \y_{E_c^k} + \A_{I_0^k}^{\T} \y_{I_c^k},
\hspace{0.1in} \y_{I_c^k} \ge \0.
\]
The facet pivot simplex method will then select an entering row 
(facet) from $\A_{N^k}$ to replace one of the rows (facets) in 
$\A_{B^k}$ to form an updated base $\A_{B^{k+1}}$ such that an updated 
$\x^{k+1}$ and an updated $\y_c^{k+1}=(\y_{E_c^{k+1}}, \y_{I_c^{k+1}})$
satisfy
\[
\A_{B^{k+1}} \x^{k+1} = \b_{B^{k+1}}, \hspace{0.1in} 
\c=\A_{E_0^{k+1}}^{\T} \y_{E_c^{k+1}} + \A_{I_0^{k+1}}^{\T} 
\y_{I_c^{k+1}}, \hspace{0.1in} 
\y_{I_c^{k+1}} \ge \0. 
\]
This process is repeated until a basic feasible solution is found. 
The basic feasible solution is, based on Theorem~\ref{optimality}, 
an optimal solution of (\ref{stdGeneralLP}).
We will show that this process will be ended in finite steps,
i.e., the algorithm will find the optimal solution in finite iterations.

The next theorem states that if an optimal solution exists, then
a basic optimal solution exists, which means that the process 
described above is indeed well-defined.

\begin{theorem}
For the standard general form of LP (\ref{stdGeneralLP}),
the following claims hold.
\begin{itemize}
\item[(i)] If there is a feasible solution of (\ref{stdGeneralLP}), 
then there is a basic feasible solution of (\ref{stdGeneralLP}).
\item[(ii)] If there is an optimal solution of (\ref{stdGeneralLP}), 
then there is a basic optimal solution of (\ref{stdGeneralLP}).
\end{itemize}
\label{fundamental}
\end{theorem}
\begin{proof}
Without the loss of generality, we may assume that $\A_E$ is full row rank
(otherwise, we may remove the redundant rows of $\A_E$) and $m<d$. 
By our notation, the number of rows of $\A_I$ is $n+2d$.
Assume that $\x$ is a feasible solution, it must have $\A_E \x = \b_E$,
and $\A_I \x \ge \b_I$. Assume further that there are $p$ active inequality
constraints at $\x$ with $p \in \{ 0, 1, \ldots, n+2d \}$, i.e., $\a_j \x = b_j$
for $j \in \mathcal{I}$ ($\A_I \x \ge \b_I$) and $j \le p$. Denote by 
$\mathcal{P} \subset \mathcal{I}$ the index set of the active inequality
constraint set, by $\mathcal{Q} \subset \mathcal{I}$ the index set of 
the inactive inequality constraint set, by by $\A_P$ the rows corresponding to
$\a_j \x = b_j$ with $j \in \mathcal{P}$, and by $\A_Q$ the rows
corresponding to $\a_j \x > b_j$ with $j \in \mathcal{Q}$. 
Let $\A_{\pi}=(\A_E, \A_P)$ be the stacked matrix of $\A_E$ and 
$\A_P$.  If $rank(\A_{\pi}) = m+p \ge d$, then $\x$ is a basic solution. 
Since $\x$ is feasible, it is a basic feasible solution.

Now assume $rank(\A_{\pi})= m+p<d$, then, there is a d-dimensional 
vector $\y \neq 0$ such that $\A_{\pi}^{\T} \y =0$. Let 
$\mathcal{R}=\mathcal{E} \cup \mathcal{P}$. Clearly, for $k \in \mathcal{R}$ 
and $\lambda \ge 0$, it must have $\a_k (\x+\lambda \y) = b_k$. 
Since $\x$ is feasible, for $j \in \mathcal{Q}$, it must
have $\a_j (\x+\lambda \y) > b_j$ for small $\lambda$. 
Assume there is at least one $j \in \mathcal{Q}$ such that 
$\a_j (\x+\lambda \y)$ decreases as $\lambda$ increases 
(otherwise consider $\a_j (\x-\lambda \y)$). Denote $\b_Q$ and
$\b_R$ the seb-vectors of $\b$ corresponding to $\A_Q$ and $\A_R$, and 
\[
\lambda^*= \{ \lambda ~|~ \a_j (\x+\lambda \y) = b_j,~j \in 
\mathcal{Q}, \hspace{0.1in} \A_Q(\x+\lambda \y)  \ge \b_Q \}.
\]
Note that $\x +\lambda^* \y$ is a feasible solution. We can move 
$j$ with $\a_j (\x+\lambda^* \y) = b_j$ from the set $\mathcal{Q}$
to $\mathcal{P}$ and then update the index sets $\mathcal{Q}$ and 
$\mathcal{P}$. Repeating this process, we will get a basic feasible solution 
of (\ref{stdGeneralLP}) due to Assumption 2 made at the end of Section
\ref{proposedAlgo}. This proves (i).

To prove (ii), we assume $\x$ is an optimal solution. If 
$rank(\A_{\pi}) = m+p \ge d$, then $\x$ is a basic optimal solution
because $\A_{\pi}\x=\b_{\pi}$ ($\b_{\pi}$ is the sub-vector of $\b$
corresponding to $\A_{\pi}$).
If $rank(\A_{\pi}) = m+p < d$, using the process as of (i), we have 
$\x+\lambda \y$ is a feasible solution for $\lambda \in (-\epsilon, \epsilon )$ 
with small $\epsilon$ because $\A_R(\x+\lambda \y) = \b_R$ and 
$\A_Q(\x+\lambda \y) > \b_Q$. If $\c^{\T} \y <0$ and $\epsilon >0$, then
\[
\c^{\T} (\x+\epsilon \y) = \c^{\T} \x + \epsilon \c^{\T} \y < \c^{\T} \x.
\]
This contradicts the assumption that $\x$ is an optimal solution of 
(\ref{stdGeneralLP}). If $\c^{\T} \y >0$ and $\epsilon >0$, then 
$\c^{\T} (\x-\epsilon \y) < \c^{\T} \x$, this again contradicts that
$\x$ is an optimal solution of (\ref{stdGeneralLP}).
Therefore, $\c^{\T} \y =0$ must hold if $\x$ is an optimal solution of
(\ref{stdGeneralLP}). This shows that $\x+\lambda \y$ is an optimal
solution. Repeating the arguments above and the process as of (i) 
proves (ii).
\hfill \qed
\end{proof}

Now we are ready to describe the details of the facet pivot simplex algorithm.

\section{The facet pivot simplex algorithm}

This section describes major steps of the facet pivot simplex
algorithm. It is largely based on the ideas of Liu et. al. 
\cite{ltz21} but has corrections, improvements, and additional materials.
First, we make an additional assumption.

\vspace{0.1in}
{\bf Assumption}
\begin{itemize}
\item[3.] For $k \geq 0$, $\A_{B^k}$ is full rank.
\end{itemize}
Clearly, this assumption is true for $k=0$ ($\A_{B^0} = \E$) 
and we will show in Theorem \ref{independent} that the assumption 
holds for $k>0$.

\subsection{Initial point}

As we have explained, the initial base can be taken 
directly from the standard general form with 
$\A_{B^0} \x^0=\E \x^0 =\b_L$, and 
$\c=\A_{I_0}^{\T} \y_{I_0}=\E^{\T} \y_{I_0} $,
and $\y_{I_0}=\bar{\c} \ge \0$. Clearly,
$rank(\A_{B^0})=rank(\E)=rank(\A_{I_0})=d$ 
and $\x^0$ can easily be obtained by solving 
$\E \x =\b_L$ because $\E$ is a full rank diagonal matrix
whose diagonal elements are either $1$ or $-1$. Unlike
vertex pivot method, facet pivot does not have an
expensive Phase I.

\subsection{Criterion to check for optimal solution}
\label{criterion}

The facet pivot simplex algorithm maintains two properties: (a) every iterate 
$\x^k$ for $k \ge 0$ is a basic solution, and (b) the condition given in
(\ref{optimal}) holds at all iterations. From (\ref{sigma}), if $\sigma_i (\x^k) = 0$,
$\sigma_j  (\x^k) \ge 0$, $\underline{\sigma}_k (\x^k) \ge 0$, and 
$\overline{\sigma}_k (\x^k) \ge 0$ hold, then, $\x^k$ is feasible.
Therefore, according to Theorem \ref{optimality}, an optimal solution is found. 
If at least one of the conditions in (\ref{sigma}) is not met,
$\x^k$ is infeasible, then the iteration will continue.

\subsection{Remove redundant constraints}

Before we find the entering and leaving rows (facets) to update the 
base of problem (\ref{stdGeneralLP}), it may be beneficial to remove
redundant constraints to simplify the problem. Let 
$r \in \mathcal{I}_1 \cup \mathcal{E}_1$, since the base
matrix $\A_{B^{k}}$ is full rank, we can represent
any non-base row (facet) $\a_r$ in $\A_N$ as
\begin{equation}
\a_r^{\T}=\sum_{j \in B^k} y_{rj} \a_j^{\T} = \A_{B^k}^{\T}\y_r
:= \A_{I_0^{k}}^{\T}\y_{I_r} 
+\A_{E_0^{k}}^{\T}\y_{E_r},
\label{nonBaseD}
\end{equation}
where the base matrix $\A_{B^{k}}$ at the $k$-th iteration is 
partitioned as inequality constraints $\A_{I_0^{k}}$ and 
equality constraints $\A_{E_0^{k}}$, i.e.,
$\A_{B^k}^{\T}= \left[ \A_{I_0^{k}}^{\T}, ~~\A_{E_0^{k}}^{\T}\right]$. 

\begin{theorem}
Assume that $\x^k$ is the basic solution of (\ref{stdGeneralLP}) 
at the $k$-th iteration and $\a_r \in \A_N$ is expressed as 
(\ref{nonBaseD}), then, the following claims hold. 
\begin{itemize}
\item[(1).] If the row vector $\a_r$ corresponds to an equality 
constraint which satisfies $\a_r \x^k = b_r$ and $y_{rj} =0$ for all 
$j \in \mathcal{I}_0$, i.e., $\y_{I_r} =0$, then, 
$\a_r \x^k = b_r$ is a redundant constraint.
\item[(2).] If the row vector $\a_r$ corresponds to an inequality 
constraint which satisfies $\a_r \x^k > b_r$, and 
$y_{rj} \ge 0$ for $\forall j \in \mathcal{I}_0$, i.e., $\y_{I_r} \ge 0$, 
then, $\a_r \x^k \ge b_r$ is a redundant constraint.
\end{itemize}
\label{redundantEqRow}
\end{theorem}
\begin{proof}
We first show that condition (1) implies the equality 
constraint $\a_r \x =b_r$ is redundant.
Since $\y_{I_r} =0$, from (\ref{nonBaseD}), we have 
$\a_r^{\T}=\A_{E_0^{k}}^{\T}\y_{E_r}$.
Multiplying both sides of (\ref{nonBaseD}) by $\x^k$ yields
$b_r=\a_r \x^k = \y_{E_r}^{\T}\A_{E_0^{k}} \x^k
=\y_{E_r}^{\T} \b_{E_0^{k}}$. Therefore,
$\y_{E_r}^{\T} \left[ \A_{E_0^{k}},~\b_{E_0^{k}} \right]
=[\a_r,~ b_r]$. This proves that if condition (1) holds, the
equality constraint $\a_r \x =b_r$ can always be expressed as
a combination of the constraints $\A_{E_0^{k}} \x = \b_{E_0^{k}}$,
therefore, it is redundant. 
Note that Condition (2) is equivalent to say that system (ii) of Corollary
\ref{farkasC} is true, this means that system (i) of Corollary
\ref{farkasC} is not true.  Let $\x^k$ be the basic solution at iteration $k$, 
since $\A_{I_0^{k}} \x^k = \b_{I_0}$ and $\A_{E_0^{k}} \x^k =\b_{E_0}$,
we must have $\a_r \bar{\x} \ge \a_r \x^k > b_r$ for any feasible 
solution $\bar{\x}$ of (\ref{stdGeneralLP}). This shows that
the inequality constraint $\a_r  \bar{\x} > b_r$
holds for all feasible solutions $\bar{\x}$, Therefore, it 
is redundant.
\hfill \qed
\end{proof}


Although this theorem provides a simple scheme to remove the redundant constraints, our computational experience shows the scheme is still expensive. Therefore, we only check if the equality constraints are
redundant. If they are, we remove the redundant ones
to make sure the base matrix $\A_B$ is independent. 

\subsection{General rules on entering/leaving row (facet) selection}

The proposed facet pivot simplex method is based on Theorem \ref{optimality},
which keeps all the iterates to meet conditions (\ref{optimal}) and
\begin{equation}
\A_{B^k} \x^k = \b_{B^k},
\label{basicSolution}
\end{equation}
the goal is to find a feasible solution of $\x$ by iteration. 
This implies that equality constraints should be selected to the 
base before the inequality constraints are selected. Therefore, 
the first rule in considering the entering row (facet) is to select 
rows (facets) of $\A_E$ with $\sigma_i \neq 0$ before the rows 
(facets) of $\A_I$ with $\sigma_j <0$. Once the rows (facets) 
of $\A_E$ are selected to the base, they will never leave the base.

\begin{remark}
In traditional vertex pivot simplex method, we cannot determine what 
columns will be part of the optimal base until the optimal
base is found. Therefore, some optimal 
columns enter and leave the base multiple times during the 
iteration. This wastes a lot of computational time due to  lack of 
intuition and using short sighted strategies \cite{yang20a}. 
In contrast, facet pivot simplex method does know that all 
equality constraints must be in the base, therefore improve computational efficiency.
\end{remark}

\subsection{Specific rules for entering row (facet) selection}

Several rules are proposed in \cite{ltz21}. Since we would like
to increase feasibility or identify infeasibility as soon as possible,
the first specific entering rule is 

\noindent
{\bf The maximal deviation rule:} Among all the rows (facets) in the non-basic
equality/inequality constraints, select the row (facet) $\a_p$ in $\A_N$ 
which has the maximal deviation from the constraint, i.e., 
\begin{equation}
|\sigma_p| = \max \{ |\sigma_i|, |\sigma_j|, 
| \underline{\sigma}_k|, | \overline{\sigma}_k|  
~\big|~ \sigma_i \neq 0, ~ \sigma_j <0, 
~ \underline{\sigma}_k <0, ~ \overline{\sigma}_k <0,
~ i,j,k \in \mathcal{I}_1 \cup \mathcal{E}_1 \}.
\label{enteringV1}
\end{equation} 

The second specific entering rule is

\noindent
{\bf The maximal normalized deviation rule:} Among all the rows (facets) in the non-basic
equality/inequality constraints, select the row (facet) $\a_p$ in $\A_N$ 
which has the maximal normalized deviation from the constraint, i.e., 
\begin{eqnarray}
|\sigma_p|/\| \a_p \| = \max \{ |\sigma_i|/\| \a_i \|, 
|\sigma_j|/\| \a_j \|, | \underline{\sigma}_k|/\| \a_k \|, 
| \overline{\sigma}_k| /\| \a_k \|     \nonumber \\
~\big|~ \sigma_i \neq 0, ~ \sigma_j <0, 
~ \underline{\sigma}_k <0, ~ \overline{\sigma}_k <0,
~ i,j,k \in \mathcal{I}_1 \cup \mathcal{E}_1 \}.
\label{enteringV2}
\end{eqnarray} 
It is worthwhile to mention that this rule finds the facet that has
the maximum distance from the current iterate to all 
infeasible constraints.

The third specific entering rule is

\noindent
{\bf The least/lowest index rule:} Assume that all equality constraints
have been selected. Among all the vectors in the non-basic
inequality constraints, select the row (facet) $\a_p$ in $\A_N$ 
which has the least/lowest index in $\mathcal{I}_1$, i.e., 
the least/lowest index in the following set
\begin{eqnarray}
\{ \sigma_j <0, ~ \underline{\sigma}_k <0, 
~ \overline{\sigma}_k <0,~j,k \in \mathcal{I}_1 \}.
\label{enteringV3}
\end{eqnarray} 

\begin{remark}
Assume that the optimal solution is not found, as 
discussed at the end of the section \ref{criterion}, at least one
of the relations $\sigma_i  \neq 0, ~ \sigma_j <0, 
~ \underline{\sigma}_k <0, ~ \overline{\sigma}_k <0$
holds. Therefore, the rules based on  (\ref{enteringV1}),
(\ref{enteringV2}) and (\ref{enteringV3}) are well-defined.
\end{remark}

\begin{remark}
The ideas of the first two entering rules are to examine 
the most restrictive constraints so that (a) we can remove 
as many redundant constraints as possible using (\ref{enteringV})
and Theorem \ref{redundantEqRow} (if an entering facet is redundant
and removed, we select another one), or (b) we can identify the
infeasibility as early as possible (we will discuss this in Section 
\ref{identifyInfeasibility}).
\end{remark}

Since $\A_{B^k}$ is full rank, then the candidate entering row 
(facet) can be expressed as
\begin{equation}
\a_p^{\T}=\sum_{j \in B^k} y_{pj}^k \a_j^{\T} 
= \A_{B^k}^{\T}\y_p^k := \A_{I_0^{k}}^{\T}\y_{I_p} 
+\A_{E_0^{k}}^{\T}\y_{E_p}.
\label{enteringV}
\end{equation}
Again, the base matrix $\A_{B^{k}}$ at iteration $k$ is 
partitioned into inequality constraints $\A_{I_0^{k}}$ and 
equality constraints $\A_{E_0^{k}}$ parts, i.e.,
$\A_{B^k}^{\T}= \left[ \A_{I_0^{k}}^{\T}, ~~\A_{E_0^{k}}^{\T}\right]$. 

\begin{remark}
In view of (\ref{enteringV1}) and (\ref{enteringV2}), the selected
entering row (facet) $\a_p$ must meet one of the following 
conditions but not both.
\begin{itemize}
\item[1.] $\a_p \x^k <b_p$ with $p \in \mathcal{I}_1 \cup \mathcal{E}_1$.
\item[2.]$\a_p \x^k  > b_p$ with $p \in \mathcal{E}_1$.
\end{itemize}
\label{pCondition}
\end{remark}

\subsection{Identify infeasible solution}\label{identifyInfeasibility}

After an entering row (facet) is selected and $\y_p$ is obtained by solving the
linear systems of equations (\ref{enteringV}), Problem (\ref{stdGeneralLP}) 
can be checked for infeasibility by the following theorem.

\begin{theorem}
Let $B^k$ be the base of (\ref{stdGeneralLP}) at the $k$-th iteration, 
denote $\x^k$ the basic (but infeasible) solution of (\ref{stdGeneralLP}), 
i.e., $\A_{B^k} \x^{k} =\b_{B^k}$. If either the condition set
\begin{itemize}
\item[(1).] (a) $\a_{p} \x^k< b_p$ for the entering row (facet)
$p \in \mathcal{I}_1 \cup \mathcal{E}_1$, and 
\newline 
(b) $y_{pj}^k \le 0$ for all $j \in \mathcal{I}_0$ in (\ref{enteringV})
\item[] or the condition set
\item[(2).] (a) $\a_{p} \x^k > b_p$ for the entering row (facet)
$p \in  \mathcal{E}_1$, and 
\newline 
(b) $y_{pj}^k \ge 0$ for all $j \in \mathcal{I}_0$ in (\ref{enteringV})
\end{itemize}
holds, then, there is no feasible solution for Problem 
(\ref{stdGeneralLP}).
\label{infeasible}
\end{theorem}
\begin{proof}
Assume that condition (1.b) holds, then, (\ref{enteringV}) can be
written as
\[
\a_p^{\T}=\A_{E_0^k}^{\T}\y_{E_p} -\A_{I_0^k}^{\T}\y_{I_p},
\hspace{0.1in} \y_{I_p} \ge \0
\]
which is equivalent to the claim that
system (ii) of Theorem \ref{farkasD} is true, this means that 
system (i) of Theorem \ref{farkasD} does not hold. Because 
$\A_{B^k} \x^{k} =\b_{B^k}$ implies that
$\A_{E_0^k} \x^{k} =\b_{E_0^k}$ and 
$\A_{I_0^k} \x^{k} \ge\b_{I_0^k}$ hold, 
it must have $\a_{p} \x^k \ge \a_p \bar{\x}$ for the entering row (facet)
$p \in \mathcal{I}_1 \cup \mathcal{E}_1$
(where $\bar{\x}$ is any feasible solution of (\ref{stdGeneralLP}) as assumed in Theorem \ref{farkasD}). 
Using assumption (1.a)
$b_p > \a_{p} \x^k$ for the entering row (facet) 
$p \in \mathcal{I}_1 \cup \mathcal{E}_1$, 
$b_p > \a_{p} \x^k \ge \a_p \bar{\x}$ must hold, i.e., there is 
no feasible solution for Problem (\ref{stdGeneralLP}). 
This proves part (1).

To prove part (2), assume that condition (2.b) holds, then, 
(\ref{enteringV}) can be written as
\[
\a_p^{\T}=\A_{E_0^k}^{\T}\y_{E_p} +\A_{I_0^k}^{\T}\y_{I_p},
\hspace{0.1in} \y_{I_p} \ge \0.
\]
Further, since $\y_{I_0} \ge \0$, in view of Corollary \ref{farkasC}, 
system (ii) of Corollary \ref{farkasC} is true, which means that 
system (i) of Corollary \ref{farkasC} does not hold. Because 
$\A_{B^k} \x^{k} =\b_{B^k}$ implies that
$\A_{E_0^k} \x^{k} =\b_{E_0^k}$ and 
$\A_{I_0^k} \x^{k} \ge\b_{I_0^k}$ hold; 
it must have $\a_{p} \x^k \le \a_p \bar{\x}$ for the entering row (facet)
$p \in \mathcal{E}_1$. Using assumption (2.a)
$b_p < \a_{p} \x^k$ for the entering row (facet) $p \in \mathcal{E}_1$, 
$b_p < \a_{p} \x^k \le \a_p \bar{\x}$ must hold, i.e., there is 
no feasible solution for Problem (\ref{stdGeneralLP}).
\hfill \qed
\end{proof}

If Problem (\ref{stdGeneralLP}) does not have a feasible solution, 
the algorithm will stop here. Assuming that the problem has a 
feasible solution, then move forward to select the 
leaving row (facet). According to Remark \ref{pCondition} and Theorem
\ref{infeasible}, the following two scenarios must be considered;
\begin{itemize}
\item[1.] $\a_p \x^k <b_p$ with 
$p \in \mathcal{I}_1 \cup \mathcal{E}_1$ and 
there is at least one $y_{pj}^k>0$ for $j \in \mathcal{I}_0$.
\item[2.] $\a_p \x^k > b_p$ with $p \in \mathcal{E}_1$ and 
there is at least one $y_{pj}^k<0$ for $j \in \mathcal{I}_0$.
\label{leavingCondition}
\end{itemize}

\subsection{Rules for leaving row (facet) selection}
\label{leavingRowSel}

As mentioned earlier, the equality constraints will never
leave the base once they are in the row (facet) base. Therefore,
the leaving row (facet) is always selected from inequality constraints.
Since $\A_{B^k}$ is full rank, denote
\begin{equation}
\c =\sum_{j \in B^k} y_{cj}^k \a_j^{\T}
=\A_{I_0^k}^{\T}\y_{I_c^k} +\A_{E_0^k}^{\T}\y_{E_c^k},
\label{oldC}
\end{equation}
where $\y_{E_c^k}$ are the coefficients corresponding to the equality 
constraints in $B^k$ and $\y_{I_c^k} \ge \0$ are the coefficients
corresponding to the inequality constraints in $B^k$.  
From the selection of the initial 
base, $\y_{I_c^0} =\bar{\c} \ge \0$ holds. According to 
Theorem \ref{optimality}, the leaving row (facet) should be selected to 
maintain $\y_{I_c^k}  \ge \0$ in all iterations $k \ge 0$ so that if 
a feasible solution is found, then, an optimal solution is found
as well according to Theorem \ref{optimality}. 
Let $q \in \mathcal{I}_0$ be the index of the leaving row (facet), 
the index of the new row (facet) base can be expressed as
\begin{equation}
B^{k+1} = B^k \cup \{p\} \setminus \{q\}.
\label{newBase}
\end{equation}
Therefore, from (\ref{enteringV}), the leaving row (facet) $\a_q$ can be 
expressed using the entering row (facet) $\a_p$ and the rest rows (facets)
in the base $B^k$ as follows:
\begin{eqnarray}
\a_q^{\T} & = & \frac{1}{y_{pq}^k}\a_p^{\T} 
+ \sum_{j \in B^k \setminus \{q\} }
\left(-\frac{y_{pj}^k}{y_{pq}^k} \right) \a_j^{\T}
=\A_{I_0^{k+1}}^{\T}\y_{I_q^{k+1}} 
+\A_{E_0^{k+1}}^{\T}\y_{E_q^{k+1}}
\nonumber \\
& := & \sum_{j \in B^{k+1}} \a_j^{\T}  y_{qj}^{k+1}
=\A_{B^{k+1}}^{\T} \y_q^{k+1}
\label{leavingV}
\end{eqnarray}
where the base matrix $\A_{B^{k+1}}$ at the $(k+1)$-th iteration is 
partitioned into inequality constraints $\A_{I_0^{k+1}}$ and 
equality constraints $\A_{E_0^{k+1}}$, $\y_{I_q^{k+1}}$ and 
$\y_{E_q^{k+1}}$ are corresponding to the inequality constraints 
and the equality constraints in $B^{k+1}$. Substituting (\ref{leavingV}) 
into (\ref{oldC}) yields
\begin{eqnarray}
\c &=&y_{cq}^{k} \a_q^{\T} + 
\sum_{j \in B^k \setminus \{q\}} y_{cj}^{k} \a_j^{\T}
\nonumber \\
&=& \frac{y_{cq}^{k}}{y_{pq}^{k}}\a_p^{\T}  -
 \sum_{j \in B^k \setminus \{q\}} 
y_{pj}^{k} \frac{y_{cq}^{k}}{y_{pq}^{k}}\a_j^{\T}
+ \sum_{j \in B^k \setminus \{q\}} y_{cj}^{k} \a_j^{\T}
\nonumber \\
&=& \frac{y_{cq}^{k}}{y_{pq}^{k}}\a_p^{\T} 
+ \sum_{j \in B^k \setminus \{q\}}
\left( y_{cj}^{k} -y_{pj}^{k} \frac{y_{cq}^{k}}{y_{pq}^{k}} \right)\a_j^{\T}
\label{newC} \\
&=& \sum_{j \in B^{k+1}} y_{cj}^{k+1} \a_j^{\T}
:=\A_{B^{k+1}}^{\T} \y_c^{k+1}
\nonumber \\
&:=& \A_{I_0^{k+1}}^{\T}\y_{I_c^{k+1}} 
+\A_{E_0^{k+1}}^{\T}\y_{E_c^{k+1}}.
\label{newCinMatrix}
\end{eqnarray}
Again, in (\ref{newCinMatrix}), the base matrix 
$\A_{B^{k+1}}$ at the $(k+1)$-th iteration is 
partitioned into inequality constraints $\A_{I_0^{k+1}}$ part and 
equality constraints $\A_{E_0^{k+1}}$ part. As discussed before, 
$\y_{I_c^{k+1}} \ge 0$ should be maintained. The
discussion is divided into the two cases described at the end of the previous
section.

{\it Case 1:}  Assume that $\a_p \x^k <b_p$ with 
$p \in \mathcal{I}_1 \cup \mathcal{E}_1$ and 
there is at least one $y_{pj}^k>0$ for $j \in \mathcal{I}_0$.

For this case, the leaving row (facet) $\a_q$ is selected to satisfy the 
condition $y_{pq}^k > 0$ and the following rule:
\begin{equation}
\frac{y_{cq}^k}{y_{pq}^k} = \min \Bigl\{ \frac{y_{cj}^k}{y_{pj}^k} 
~\bigg|~  y_{pj}^k>0, \hspace{0.1in} j \in \mathcal{I}_0 \Bigr\}.
\label{leavingRule}
\end{equation}
Since $q \in \mathcal{I}_0$, it must have $y_{cq}^k\ge 0$, which 
means $\frac{y_{cq}^k}{y_{pq}^k} \ge 0$ because $y_{pq}^k>0$. Also, 
it must have $\left( y_{cj}^k -y_{pj}^k \frac{y_{cq}^k}{y_{pq}^k} \right) \ge 0$
for all $j \in \mathcal{I}_0$ because of (\ref{leavingRule}).
This indicates that, according to (\ref{newC}), 
$\y_{I_c^{k+1}} \ge 0$ and condition (\ref{optimal}) holds. 
The following theorem reveals several important facts.

\begin{theorem}[\cite{ltz21}]
Let $B^k$ be the base of (\ref{stdGeneralLP}) in the $k$-th iteration. 
Denote by $\x^k$ the basic solution of (\ref{stdGeneralLP})
corresponding to $B^k$, i.e., $\A_{B^k} \x^{k} =\b_{B^k}$, 
and by $\x^{k+1}$ the basic solution of (\ref{stdGeneralLP}) 
corresponding to $B^{k+1}$, i.e., 
$\A_{B^{k+1}} \x^{{k+1}} =\b_{B^{k+1}}$. Assume that 
the entering row (facet) is $\a_{p}$, 
$p \in \mathcal{I}_1^k \cup \mathcal{E}_1^k$;
the leaving row (facet) is $\a_{q}$, $q \in \mathcal{I}_0^k$; and the 
following conditions hold
\begin{itemize}
\item[(a)] $\a_{p} \x^k< b_p$, 
\item[(b)] there is an index $j \in \mathcal{I}_0^k$ such that $y_{pj}^k>0$, and 
\item[(c)] the leaving row (facet) $q$ is determined by (\ref{leavingRule}),
\end{itemize}
then, 
\begin{itemize}
\item[(i)] $\c$ is given as (\ref{newC}) with 
$\left( y_{cj}^k -y_{pj}^k \frac{y_{cq}^k}{y_{pq}^k} \right) \ge 0$
for all $j \in \mathcal{I}_0^k$. Moreover, $\y_{I_c^{k+1}} \ge \0$
holds.
\item[(ii)] The objective function is monotonically increasing, i.e.,
\begin{equation}
\c^{\T} \x^{k+1} -\c^{\T} \x^{k}
= \frac{y_{cq}^k}{y_{pq}^k} 
\left( b_p-\a_p \x^{k}  \right)  \ge 0.
\label{diffC}
\end{equation}
\item[(iii)] Facet $q$ is strictly feasible in the next iteration, i.e.,
\begin{equation}
\a_q \x^{k+1} >b_q.
\label{aqX1}
\end{equation}
\item[(iv)] If $y_{pq}^k>0$ and $y_{pj}^k \le 0$ for all 
$j \in \mathcal{I}_0^k \setminus \{ q \}$, then,
$\a_{q} \x \ge b_q$ is a redundant constraint. 
\end{itemize}
\label{redundant}
\end{theorem}
\begin{proof}
Most parts of (i) have been proved before this theorem. Since
the leaving row (facet) is an inequality constraint, from (\ref{oldC}),
it follows $y_{cq}^k \ge 0$. This shows that $\y_{I_c^{k+1}} \ge \0$
and therefore proves part (i). Since $\x^{k+1}$ is a basic solution,
from (\ref{newC}), then 
\begin{eqnarray}
\c^{\T} \x^{k+1} & = &
\frac{y_{cq}^k}{y_{pq}^k}\a_p  \x^{k+1}
+ \sum_{j \in B^k \setminus \{q\}}
\left( y_{cj}^k -y_{pj}^k \frac{y_{cq}^k}{y_{pq}^k} \right)\a_j \x^{k+1}
\nonumber \\
&=& \frac{y_{cq}^k}{y_{pq}^k} b_p
+ \sum_{j \in B^k \setminus \{q\}}
\left( y_{cj}^k -y_{pj}^k \frac{y_{cq}^k}{y_{pq}^k} \right) b_j
\nonumber \\
&=& \frac{y_{cq}^k}{y_{pq}^k} b_p + \sum_{j \in B^k }
\left( y_{cj}^k -y_{pj}^k \frac{y_{cq}^k}{y_{pq}^k} \right) b_j.
\label{tmp1}
\end{eqnarray}
The last equation holds because
$y_{cq}^k -y_{pq}^k \frac{y_{cq}^k}{y_{pq}^k}=0$. From (\ref{oldC}), 
it follows
\begin{eqnarray}
\c^{\T} \x^{k} & = & \sum_{j \in B^k } y_{cj}^k \a_j \x^{k}
=\sum_{j \in B^k } y_{cj}^k b_j.
\label{tmp2}
\end{eqnarray}
Subtracting (\ref{tmp2}) from (\ref{tmp1}) and invoking 
(\ref{enteringV}) yield
\begin{eqnarray}
\c^{\T} \x^{k+1} - \c^{\T} \x^{k} & = &
\frac{y_{cq}^k}{y_{pq}^k} \left(  
b_p- \sum_{j \in B^k }y_{pj}  b_j\right)
\nonumber \\
&=& \frac{y_{cq}^k}{y_{pq}^k} \left(  
b_p- \sum_{j \in B^k }y_{pj}^k \a_j  \x^{k} \right)
\nonumber \\
&=&  \frac{y_{cq}^k}{y_{pq}^k} \left(  
b_p- \a_p \x^{k} \right) \ge 0,
\label{tmp3}
\end{eqnarray}
the last inequality follows from assumption (a) and ${y_{cq}^k} \ge 0$. 
This proves part (ii). Multiplying both sides of (\ref{leavingV}) by 
$\x^{k+1}$, using (\ref{enteringV}), assumption (c) ($y_{pq}^k>0$), 
and assumption (a) ($b_p - \a_p\x^{k}>0$) yield
\begin{eqnarray}
\a_q \x^{k+1} & = & \frac{1}{y_{pq}^k} \a_p  \x^{k+1} 
+ \sum_{j \in B^k \setminus \{q\}}
\left(-\frac{y_{pj}^k}{y_{pq}^k} \right) \a_j   \x^{k+1} 
\nonumber \\
&=&  \frac{1}{y_{pq}^k}  b_p + \sum_{j \in B^k \setminus \{q\}}
\left(-\frac{y_{pj}^k}{y_{pq}^k} \right)  b_j
\nonumber \\
&=&  \frac{1}{y_{pq}^k}  \left( b_p - \sum_{j \in B^k }
{y_{pj}^k}b_j \right) + \frac{y_{pq}^k}{y_{pq}^k} b_q
\nonumber \\
&=&   \frac{1}{y_{pq}^k}  \left( b_p - \a_p\x^{k} \right) + b_q
>b_q.
\label{tmp4}
\end{eqnarray}
This proves (iii).
In view of the condition in (iv) and (\ref{leavingV}), it follows that
\[
\A_{E_0^{k+1}}^{\T}\y_{E_q^{k+1}}
+\A_{I_0^{k+1}}^{\T}\y_{I_q^{k+1}} =\a_q^{\T},
~~~\y_{I_q^{k+1}} \ge \0
\]
holds. This indicates that system (ii) of Corollary \ref{farkasC} 
holds,  therefore, system (i) of Corollary \ref{farkasC} is not true, 
i.e., at least one of the following relations does not hold for any
feasible solution $\x$
\[
\A_{I_0^{k+1}} \x^{k+1} \ge \b_{I_0^{k+1}},~~\A_{E_0^{k+1}} \x^{k+1}
 = \b_{E_0^{k+1}}, ~~\a_{q} \x <\a_q {\x}^{k+1} .
\]
The first two relations hold because $\x^{k+1}$ is a basic solution, 
it must have $\a_{q} \x \ge \a_q {\x}^{k+1}$. In view of 
(\ref{tmp4}), $\a_q \x^{k+1} >b_q$, this shows that 
$\a_{q} \x >b_q$ holds for all feasible $\x$, therefore, the 
constraint is redundant. This proves part (iv).
\hfill \qed
\end{proof}

{\it Case 2:}  Assume that $\a_p \x^k > b_p$ with 
$p \in \mathcal{E}_1$ and 
there is at least one $y_{pj}^k<0$ for $j \in \mathcal{I}_0$.

For this case, the leaving row (facet) $\a_q$ is selected to satisfy the 
condition $y_{pq}^k < 0$ and the following rule:
\begin{equation}
\frac{y_{cq}^k}{y_{pq}^k} = \max \Bigl\{ \frac{y_{cj}^k}{y_{pj}^k} 
~\bigg|~  y_{pj}^k<0, \hspace{0.1in} j \in \mathcal{I}_0 \Bigr\}.
\label{leavingRulea}
\end{equation}
Since $q \in \mathcal{I}_0$, it must have $y_{cq}^k \ge 0$, which 
means $\frac{y_{cq}^k}{y_{pq}^k} \le 0$ because $y_{pq}^k<0$. Also, 
it must have $\left( y_{cj}^k -y_{pj}^k \frac{y_{cq}^k}{y_{pq}^k} \right) \ge 0$
for all $j \in \mathcal{I}_0$ because of (\ref{leavingRulea}).
Since $\a_p$ is an equality constraint, this indicates that, according to (\ref{newC}), 
$\y_{I_c^{k+1}} \ge 0$ and condition (\ref{optimal}) holds. 

The following theorem reveals several important facts.

\begin{theorem}[\cite{ltz21}]
Let $B^k$ be the base of (\ref{stdGeneralLP}) at the $k$-th iteration. 
Denote by $\x^k$ the basic solution of (\ref{stdGeneralLP})
corresponding to $B^k$, i.e., $\A_{B^k} \x^{k} =\b_{B^k}$, 
and by $\x^{k+1}$ the basic solution of (\ref{stdGeneralLP}) 
corresponding to $B^{k+1}$,  i.e., 
$\A_{B^{k+1}} \x^{{k+1}} =\b_{B^{k+1}}$. Assume that
the entering row (facet) is $\a_{p}$, $p \in \mathcal{E}_1^k$;
the leaving row (facet) is $\a_{q}$, $q \in \mathcal{I}_0^k$; and the 
following conditions hold
\begin{itemize}
\item[(a)] $\a_{p} \x^k> b_p$, 
\item[(b)] there is an index $j \in \mathcal{I}_0^k$ such that $y_{pj}^k<0$, and 
\item[(c)] the leaving row (facet) $\a_q$ is determined by (\ref{leavingRulea}), 
\end{itemize}
then, 
\begin{itemize}
\item[(i)] $\c$ is given as (\ref{newC})
with $\left( y_{cj}^k -y_{pj}^k \frac{y_{cq}^k}{y_{pq}^k} \right) \ge 0$
for all $j \in \mathcal{I}_0^k$. Moreover, $\y_{I_c^{k+1}} \ge \0$
holds.
\item[(ii)] The objective function is monotonically increasing, i.e.,
\begin{equation}
\c^{\T} \x^{k+1} -\c^{\T} \x^{k}
= \frac{y_{cq}^k}{y_{pq}^k} 
\left( b_p-\a_p \x^{k}  \right)  \ge 0.
\label{diffCa}
\end{equation}
\item[(iii)] Facet $q$ is strictly feasible in the next iteration, i.e.,
\begin{equation}
\a_q \x^{k+1} >b_q.
\label{aqX1a}
\end{equation}
\item[(iv)] If $y_{pq}^k<0$ and $y_{pj}^k \ge 0$ for all 
$j \in \mathcal{I}_0^k \setminus \{ q \}$, then,
$\a_{q} \x \ge b_q$ is a redundant constraint. 
\end{itemize}
\label{redundanta}
\end{theorem}
\begin{proof}
Part (i) has been proved just before this theorem. Therefore, only 
parts (ii), (iii), and (iv) are proved here. Since $\x^{k+1}$ is a basic solution, 
from (\ref{newC}), following the exact the same steps of 
the derivation of (\ref{tmp1}), the following relation can be established.
\begin{eqnarray}
\c^{\T} \x^{k+1} 
&=& \frac{y_{cq}^k}{y_{pq}^k} b_p + \sum_{j \in B^k }
\left( y_{cj}^k -y_{pj}^k \frac{y_{cq}^k}{y_{pq}^k} \right) b_j.
\label{tmp1a}
\end{eqnarray}
From (\ref{oldC}), again, it follows
\begin{eqnarray}
\c^{\T} \x^{k} & = & \sum_{j \in B^k } y_{cj}^k \a_j \x^{k}
=\sum_{j \in B^k } y_{cj}^k b_j.
\label{tmp2a}
\end{eqnarray}
Subtracting (\ref{tmp2a}) from (\ref{tmp1a}), invoking (\ref{enteringV}),
and following the exactly same steps of the derivation of (\ref{tmp3}), 
this yields
\begin{eqnarray}
\c^{\T} \x^{k+1} - \c^{\T} \x^{k} 
&=&  \frac{y_{cq}^k}{y_{pq}^k} \left(  
b_p- \a_p \x^{k} \right) \ge 0,
\label{tmp3a}
\end{eqnarray}
the last inequality follows from assumption (a) and ${y_{pq}^k}<0$. 
This proves part (ii). Multiplying both sides of (\ref{leavingV}) by 
$\x^{k+1}$, using (\ref{enteringV}), and assumptions (a) and (c), and
following the exactly same steps in the derivation of (\ref{tmp4}),
it follows
\begin{eqnarray}
\a_q \x^{k+1} 
&=&   \frac{1}{y_{pq}^k}  \left( b_p - \a_p\x^{k} \right) + b_q
>b_q,
\label{tmp4a}
\end{eqnarray}
again, the last inequality follows from assumption (a) and 
${y_{pq}^k}<0$. This proves (iii).
In view of the condition in (iv) and (\ref{leavingV}) and noticing that
$\mathcal{I}_0^{k+1}=\mathcal{I}_0^{k}\setminus \{ q \}$ (because the
entering facet is an equality constraint in this case), it follows that
\[
\A_{E_0^{k+1}}^{\T}\y_{E_q^{k+1}}
+\A_{I_0^{k+1}}^{\T}\y_{I_q^{k+1}} =\a_q^{\T},
~~~\y_{I_q^{k+1}} \ge \0
\]
holds. This indicates that system (ii) of Corollary \ref{farkasC} holds.
Therefore, system (i) of Corollary \ref{farkasC} is not true, 
i.e., at least one of the following relations does not hold for any
feasible solution $\x$
\[
\A_{I_0^{k+1}} \x^{k+1} \ge \b_{I_0^{k+1}},~~\A_{E_0^{k+1}} \x^{k+1}
 = \b_{E_0^{k+1}}, ~~\a_{q} \x <\a_q {\x}^{k+1} .
\]
The first two relations hold because $\x^{k+1}$ is a basic solution
at iteration $k+1$, it must have $\a_{q} \x \ge \a_q {\x}^{k+1}$. 
In view of (\ref{tmp4a}), $\a_q \x^{k+1} >b_q$, this shows that 
$\a_{q} \x >b_q$ holds for all feasible $\x$, therefore, the 
constraint is redundant. This proves part (iv).
\hfill \qed
\end{proof}

\begin{remark}
Once a constraint is identified as a redundant one, there
is no need to consider it in the remaining iterations. 
\end{remark}   

\begin{remark}
In case there is a tie in the selection of the leaving row (facet) 
using (\ref{leavingRule}), the row (facet) with the least/lowest index 
should be selected. 
\end{remark}

Combining the results of Theorems \ref{redundant} and \ref{redundanta}, 
we have the following corollary:
\begin{corollary}
Let $\y_{I_c^{k+1}}$ be defined in (\ref{newCinMatrix}), and
$\a_q \in B^k$ be the leaving facet at the $k$-th iteration. Then, the
following relations hold
\begin{equation}
\y_{I_c^{k+1}} \ge \0, \hspace{0.1in}
\c^{\T} \x^{k+1} - \c^{\T} \x^{k} \ge 0, \hspace{0.1in}
\a_q \x^{k+1} > \a_q \x^k = b_q.
\end{equation}
\end{corollary}

\begin{theorem}
Assume that the rows of $\A_{B^k}$ are independent, then the rows
of $\A_{B^{k+1}}$ are also independent.
\label{independent}
\end{theorem}
\begin{proof}
Denote $\u=(u_1 \ldots, u_d)$ and $\v=(v_1 \ldots, v_d)$. 
Using (\ref{enteringV}), we have the following equivalent expressions:
\begin{eqnarray}
& & \A_{B^{k+1}}^{\T} \v =0 
\nonumber \\
& \iff & \sum_{j \in B^{k+1}} \a_j^{\T}  v_j =0
\nonumber \\
& \iff & \sum_{j \in B^{k+1}\setminus \{ p \}} \a_j^{\T}  v_j
+ \a_p^{\T}  v_p =0
\nonumber \\
& \iff & \sum_{j \in B^{k+1}\setminus \{ p \}} \a_j^{\T}  v_j
+ \sum_{j \in B^k}  \a_j^{\T} y_{pj}^k v_p =0
\nonumber \\
& \iff & \sum_{j \in B^{k+1}\setminus \{ p \}} 
\a_j^{\T}  (v_j+y_{pj}^k v_p) + \a_q^{\T} y_{pq}^k v_p=0
\nonumber \\
& \iff & \sum_{j \in B^{k}} \a_j^{\T} u_j =0
\end{eqnarray}
where $u_q= y_{pq}^k v_p$, and for $j \neq q$, $u_j =v_j+y_{pj}^k v_p$.
Since the rows of $\B^k$ are independent, it follows that
$\sum_{j \in B^{k}} \a_j^{\T} u_j =0$ holds if and only if $u_j=0$
for $j \in B^{k}$. Since $y_{pq}^k>0$ and $u_q=0$ imply that 
$v_p=0$, which in turn implies the $v_j=0$ for $j=1, \ldots, d$.
Therefore, $\A_{B^{k+1}}^{\T} \v =0$ if and only if $\v =0$,
i.e., the rows of $B^{k+1}$ are also independent.
\hfill \qed
\end{proof}

\begin{remark}
We have observed that vertex pivot method may select a base
which is nearly dependent or actually dependent, which causes the
numerical problems in our testing. This theorem explains
why the facet pivot method is observed very robust.
\end{remark}

\subsection{Unbounded solution}

There are cases where linear programming problems have unbounded
solutions. The following theorem provides the criteria to identify these
cases.

\begin{theorem}[\cite{ltz21}]
Let $M$ be the artificial bound introduced in Remark \ref{1stRemark}.
If at least one basic row (facet) in $\A_B$ has an artificial  
bound $M$ or $-M$, and it is reached at the end of the iteration, then the linear 
programming problem is unbounded.
\end{theorem}
\begin{proof}
The claim is obvious and the proof is omitted.
\hfill \qed
\end{proof}

\subsection{The facet pivot simplex algorithm}

Summarizing the results discussed in this section, 
the facet pivot simplex algorithm is given as follows:
\begin{algorithm} {\ } \\
\begin{algorithmic}[1] 
\STATE Data: Matrices $\A^I$, $\A^J $, $\E$, $\F$, vectors $\b^I$, $\b^J$, 
$\u$, $\Bell$, and $\c$. 
\STATE Form the standard general LP problem and $y_{cj}^0=\bar{c}_j$.
\STATE Compute the initial basic solution $\x^0$ from  $\E\x^0=\b_L$
	(i.e., $\A_{B^0} \x^0 = \b_L$).
\STATE Compute the constraint violation determinants $\sigma_i$, 
$\sigma_j$, $\underline{\sigma}_k$, and $\overline{\sigma}_k$ using (\ref{sigma}).
\WHILE{$\sigma_i \neq 0$ or $\sigma_j < 0$ or $\underline{\sigma}_k < 0$
or $\overline{\sigma}_k < 0$}
	\IF {some equality constraints are redundant (Theorem \ref{redundantEqRow})}
		\STATE Remove the redundant equality constraints.
	\ENDIF
	\STATE Select the entering row (facet) $\a_p$ using (\ref{enteringV1}) or
	 (\ref{enteringV2}) or least/lowest index rule. Given $\a_p$, compute
	 $\y_{I_p^k}$ (i.e., $y_{pj}^k$) by solving linear systems of equations 	
	(\ref{enteringV}).
	\IF {there is no feasible solution (Theorem \ref{infeasible})}
		\STATE Exit while loop and report ``there is no feasible solution''.
	\ENDIF
	\IF {$\a_p \x^k <b_p$ with $p \in \mathcal{I}_1 \cup \mathcal{E}_1$ 
		and there is at least one $y_{pj}^k>0$ for $j \in \mathcal{I}_0$}
		\STATE Select leaving row (facet) $\a_q$ by using (\ref{leavingRule}). 
	\ELSIF {$\a_p \x^k > b_p$ with $p \in \mathcal{E}_1$ and 
		there is at least one $y_{pj}^k<0$ for $j \in \mathcal{I}_0$}
		\STATE Select leaving row (facet) $\a_q$ by using (\ref{leavingRulea}).
	\ENDIF
	\STATE Update base using (\ref{newBase}).
	\STATE Update  $\c$ (i.e., $y_{cj}^{k+1}$) using (\ref{newC}), i.e., 
		$y_{cp}^{k+1}=\frac{y_{cq}^k}{y_{pq}^k}$ and 
		$y_{cj}^{k+1}=\left( y_{cj}^k -y_{pj}^k \frac{y_{cq}^k}{y_{pq}^k} \right)$. 
	\IF {leaving row (facet) $\a_q$ is redundant (Theorems \ref{redundant} and \ref{redundanta})}
		\STATE Remove the $q$-th constraint from $\mathcal{I}$.
	\ENDIF
	\STATE Compute the updated solution $\x^{k+1}$ from  
	$\A_{B^{k+1}}\x=\b_{B^{k+1}}$
	\STATE Compute the constraint violation determinants $\sigma_i$,
	$\sigma_j$, and $\sigma_k$ using (\ref{sigma}).
	\STATE $k \Leftarrow k+1$.
\ENDWHILE
\end{algorithmic}
\label{mainAlg}
\end{algorithm}

\subsection{Finite iterations of the facet pivot simplex algorithm}

Convergence of the conventional (vertex pivot) simplex method depends 
on if cycling occurs or not, which was realized \cite{hoffman53} 
shortly after Dantzig published his seminal work. It is
well-known that the conventional vertex pivot simplex method will
find the optimal solution in finite iterations if cycling does not occur,
which is true if Bland's rule is used \cite{bland77}. 
The cycling problem for conventional vertex pivot simplex method 
has been studied by several authors, for example, 
\cite{beale55,hm04,zornig06}. 
Cycling is a phenomenon where the iterates move in 
a cycle. For the conventional vertex pivot simplex method, when 
a basic feasible solution is degenerate, 
after a few iterations using a vertex pivot simplex algorithm, it
may return to a previously constructed basic feasible 
solution. For the facet pivot simplex method, it is observed that
a set of base constraints may be repeated after 
some iterations. For both vertex and facet pivot
simplex methods, when cycling happens, there is no
change in objective function and the optimal solution
may never be reached.

Similar to Bland's rule \cite{bland77} for the conventional 
simplex method, if (a) the least/lowest index rule is applied 
in the selection of the entering constraint, and (b) the 
least/lowest index rule is applied when a tie occurs in the 
selection of leaving constraints, then the facet pivot simplex algorithm 
will find the optimal solution in finite steps. This can
be shown by the following arguments.

First, the number of bases of the linear programming Problem
(\ref{stdGeneralLP}) is finite. Let $N=m+n+2d$, and denote
the number of $d$-combinations in a set of $N$ elements as
$C(N,d)$, it is straightforward to see that the number of bases 
of Problem (\ref{stdGeneralLP}) is at most $C(N,d)$.
Second, in every iteration, we have seen from (\ref{diffC})
and (\ref{tmp3a}) that the objective function is monotonically 
non-deceasing. Third, Liu et. al. \cite{ltz21} showed the
following result which is similar to Bland's theorem.

\begin{theorem}[\cite{ltz21}]
If the least/lowest index rule is used in the selection of 
row (facet) base in Algorithm \ref{mainAlg}, then cycling will not happen. 
Therefore, the algorithm finds the optimal solution in finite iterations.
\end{theorem}

%

\section{Some implementation details and numerical test}\label{testing}

Algorithm \ref{mainAlg} has been implemented in Matlab.
Numerical tests for the proposed algorithm have been performed
for two purposes. First, verify that the facet 
pivot simplex method indeed solves some specially designed hard 
LP problems effectively, including benchmark cycling 
problems \cite{yang21} and Klee-Minty cube problems 
\cite{km72}. Second, determine if this method 
is competitive to the Dantzig's most negative pivot rule 
and dual simplex method for 
general benchmark testing LP problems, for example,
Netlib benchmark LP problems \cite{bdgr95}, as it is known 
that (a) Dantzig's most negative rule has been one of the most 
efficient deterministic pivot rules for LP problems \cite{ps14},
and (b) dual simplex method has some similarity to the facet pivot method.

\subsection{Some implementation considerations}

Some important implementation considerations are provided 
in this section so that readers can repeat the numerical tests 
reported in this paper. Since least/lowest index rule is normally 
not very efficient, the maximal deviation and maximal normalized deviation rules in Step 9 of 
Algorithm \ref{mainAlg} are implemented. All the testing results
reported in this section are based on these implementations.


To group equality constraints and inequality constraints, all
equality constraints are placed at the bottom of the base
matrix $\A_B$, and all inequality constraints are placed at the top 
of the base matrix. This makes it easy to check if $\y_{I_p} \geq \0$
in Theorem \ref{infeasible}, and if $\y_{I_c} \geq \0$ in Theorem
\ref{optimality}.

In Algorithm \ref{mainAlg}, $y_{pj}^k$ is calculated
in Step 9 by solving the linear system equations (\ref{enteringV}), 
and $\x^{k+1}$ in Step 23 by solving the linear system equations
(\ref{basicSolution}). Computation in these two steps uses
the same LU decomposition for $\A_{B^k}$, 
which will save significant amounts of CPU time in every iteration.
In addition, the basis matrix of any simplex step differs from that of the preceding step in only one row, so it is possible to make full use of the structure and to adopt Bartels and Golub update, or Forrest and Tomlin update in the LU decomposition to make the proposed algorithm more efficient. This level of details is beyond the scope of this paper.

\subsection{Test on Netlib benchmark problems with lower and upper bounds}\label{testSection}


Netlib problems have been widely used for testing linear programming
algorithms/codes, see for example,  \cite{lms92,Mehrotra92,yang17}.
For this test set, many problems have bases with poor condition numbers.
Our implementation shows that vertex simplex methods (without using
hiding tricks in commercial software) oftentimes select
a basis with very bad condition numbers and fail to find a solution.
In this section, we test the Netlib problems that have lower and upper bounds,
which is more general than the standard LP problems but less general than the
problem discussed in (\ref{generalLP}). This set of problems can be 
expressed as 
\begin{align}
\min \hspace{0.5in}  & \c^{\Tr} \x \nonumber \\ 
\mbox{\rm subject to} \hspace{0.3in}
&  \A\x = \b,~~~~~ {\Bell}\le \x \le \u. 
\label{netlibLP0}
\end{align}
While facet pivot and dual pivot algorithms can solve (\ref{netlibLP0}) directly, Dantzig's 
pivot method requires converting the problem to the following standard form:
\begin{subequations}
\begin{align}
\min \hspace{0.5in}  &  [\c^{\Tr}~\0^{\T}~\0^{\T}]
(\x,~\y ,~\z)  \\ 
\mbox{\rm subject to} \hspace{0.3in}
&  \left[ \begin{array}{cccc} 
\A & \0  & \0     \\ 
\I  &  \I  & \0    \\
\I  & \0    &   -\I   
\end{array} \right]
 \left[ \begin{array}{c}
\x \\ \y \\ \z \\  \end{array} \right]
= \left[ \begin{array}{c} 
\b  \\ \u \\ {\Bell}
\end{array} \right] \\
& ( \x, \y, \z) \ge \0. 
\end{align}
\label{netlibLP1}
\end{subequations}

\newpage
\footnotesize
\begin{longtable}{|c|c|c|c|c|c|c|c|}
\hline          
{Problem name} &  $m$ & $n$ & $d$ &  method  &  CPU  &   iter & obj 
\\ \hline
bore3d     & 233  &  0 &  334   &   max deviation 
&   0.1164  &  158 & 1.3731e+03   \\
&&&& max normalized deviation  & 0.1143 & 159  & 1.3731e+03 \\
&&&& Dantzig's simplex  & 1.4779  &  1510 &  1.3731e+03       \\
&&&& Dual simplex & * & * & * \\ \hline
capri   &  271  & 0  & 482      &  max deviation  
& 1.0281  & 592 & 2.6900E+03  \\
&&&& max normalized  deviation  & 0.7987 & 514  & 2.6900E+03 \\
&&&& Dantzig's simplex  & 5.7774  & 2940 &   2.6900E+03  \\
&&&& Dual simplex & + & + & +  \\ \hline
cre\_a  & 3516 & 0  &  7248  &  max deviation 
& 604.0044 & 5131  & 2.3595e+07  \\
&&&& max normalized deviation & 508.2350 & 4676  & 2.3595e+07 \\
&&&& Dantzig's simplex & - & 100000 & -  \\
&&&& Dual simplex & * & * & *   \\ \hline
cre\_c  & 3068 & 0  & 6411    &    max deviation
& 348.6073 & 4226 & 2.5275e+07  \\
&&&& max normalized deviation & 207.1541 & 3283  & 2.5275e+07 \\
&&&& Dantzig's simplex & - &  100000  & - \\
&&&& Dual simplex & * & * & * \\ \hline   
d6cube & 415 & 0  &  6184     &    max deviation
& 351.1482 & 1244 & 3.1549e+02  \\
&&&& max normalized deviation & 264.1741 & 1073  & 315.4917 \\
&&&& Dantzig's simplex & - & 100000 & -   \\
&&&& Dual simplex & + & + & +  \\ \hline
e226 & 223 & 0 &   472    &    max deviation 
& 1.0799 & 555 & -18.7519  \\  
&&&& max normalized deviation & 1.3605  & 697 &  -18.7519  \\
&&&& Dantzig's simplex & 3.9195  & 2464 &  -18.7519  \\
&&&& Dual simplex & + & + & +   \\ \hline
finnis & 497 & 0 & 1064  &    max deviation 
& 3.4604 &  728 & 1.7279e+05  \\
&&&& max normalized deviation & 3.5296  & 726 &  1.7279e+05  \\
&&&& Dantzig's simplex  & 54.2443   & 9098 &   1.7279e+05  \\
&&&& Dual simplex & + & + & +   \\ \hline
fit1p  & 627  & 0 & 1677    &    max deviation 
& 26.4729 & 1623 &  9.1464e+03  \\
&&&& max normalized deviation & 29.0710  & 1749 &  9.1464e+03  \\
&&&& Dantzig's simplex  & 197.5828 & 21679 &    9.1464e+03  \\
&&&& Dual simplex & + & + & +    \\ \hline
fit2p & 3000 & 0 & 13525    &    max deviation  
& 1.1779e+4 & 16618 &  6.8464e+04  \\
&&&& max normalized deviation & 1.1945e+4  & 16739 &  6.8464e+04  \\
&&&& Dantzig's simplex & - & 100000 & -     \\
&&&& Dual simplex & + & + & +    \\ \hline
ganges & 1309 & 0 & 1706   &    max deviation 
& 7.8143 & 1534 & -1.0959e+05  \\
&&&& max normalized deviation & 7.4581  & 1497 &  -1.0959e+05  \\
&&&& Dantzig's simplex & 34.3227  & 6496 &   -1.0959e+05 \\ 
&&&& Dual simplex & + & + & +      \\ \hline
gfrd\_pnc & 616 & 0   & 1160    &    max deviation 
& 1.4768 & 574 & 6.9022e+06  \\
&&&& max normalized deviation & 2.9531  & 640 &  6.9022e+06  \\
&&&& Dantzig's simplex & 29.8027  & 5229 & 6.9022e+06  \\
&&&& Dual simplex & + & + & +        \\ \hline
grow15*   & 300 & 0 & 645    &    max deviation 
 & 51.8446 & 10000  & -1.0968e+08  \\
 &&&& max normalized deviation & 9.0546  & 2367 &  -1.0687e+08  \\
&&&& Dantzig's simplex &  12.1133 & 2629 &  -1.0687e+08 \\
&&&& Dual simplex & + & + & +     \\ \hline
grow7*   & 140 & 0  &  301    &    max deviation 
& 9.0927 & 10000 & -4.8627e+07 \\  
&&&& max normalized deviation & 9.9165  & 10000 &  -4.8627e+07  \\
&&&& Dantzig's simplex &   1.4891  & 1105  & -4.7788e+07 \\
&&&& Dual simplex & + & + & +      \\ \hline
kb2   & 43 & 0 &  68  &    max deviation 
 &  0.0612 & 139 & -1.7499e+03  \\
 &&&& max normalized deviation & 0.0687  & 111 &  -1.74999  \\
&&&& Dantzig's simplex &  0.1387  & 320  &  -1.7499e+03  \\
&&&& Dual simplex & + & + & +      \\ \hline
ken\_07   &  2426 & 0 &  3602   &    max deviation 
& 113.3739 & 4321 & -6.7952e+08 \\ 
&&&& max normalized deviation & 95.9149  & 4008 &  -6.7953e+08  \\
&&&& Dantzig's simplex & - & 100000 &  - \\
&&&& Dual simplex & * & * & *    \\ \hline
pds\_02  & 2953 & 0  & 7716  &    max deviation 
& 4751 & 21339 & 2.8858e+10  \\
&&&& max normalized deviation & 24497  & 40774 &  2.8858e+10  \\
&&&& Dantzig's simplex & 670  & 41995 & 2.8858e+10  \\
&&&& Dual simplex & * & * & *    \\ \hline
recipe   & 91 & 0  & 204  &    max deviation 
& 0.0559 & 47 & -266.6160 \\
&&&& max normalized deviation & 0.1915  & 46 &  -266.6160  \\
&&&& Dantzig's simplex & 0.4704   & 682 &   -266.6160  \\
&&&& Dual simplex & 0.043 & 28 & -266.6160      \\ \hline
scorpion   & 388 & 0  & 466  &    max deviation 
& 0.2897 &  320 & 1.8781e+03 \\
&&&& max normalized deviation & 1.4751  & 329 &  1.8781e+03  \\
&&&& Dantzig's simplex & 3.3254  & 2251 &  1.6079e+12 \\
&&&& Dual simplex & * & * & *     \\ \hline
shell   & 536 & 0  & 1777  &    max deviation 
& 3.3670 & 690 & 1.2088e+09 \\
&&&& max normalized deviation & 10.3637  & 715 &  1.2088e+09  \\
&&&& Dantzig's simplex & - & 100000 & - \\
&&&& Dual simplex & - & 100000 & -     \\ \hline
sierra    & 1227 & 0  & 2735  &   max deviation 
& 7.6812 & 1266 & 1.5394e+07 \\
&&&& max normalized deviation & 74.5192  & 1347 &  1.5394e+07  \\
&&&& Dantzig's simplex & - & 100000 & -  \\
&&&& Dual simplex & * & * & *    \\ \hline
standata  & 359 & 0  & 1274  &    max deviation 
& 0.3990 & 123 & 1.2577e+03 \\
&&&& max normalized deviation & 3.1593  & 234 &  1.2577e+03  \\
&&&& Dantzig's simplex &  53.7744   & 8665 &   1.2577e+03   \\
&&&& Dual simplex & + & + & +      \\ \hline
standgub  & 361 & 0  & 1383  &   max deviation 
& 0.4033 & 123  & 1.2577e+03 \\
&&&& max normalized deviation & 3.1752  & 234 &  1.2577e+03  \\
&&&& Dantzig's simplex & 48.5099 & 8665 &  1.2577e+03  \\
&&&& Dual simplex & + & + & +       \\ \hline
standmps  & 467 & 0  & 1274  &   max deviation 
& 0.6406 & 250 & 1.4060e+03 \\
&&&& max normalized deviation & 52.6339  & 3623 &  1.4060   \\
&&&& Dantzig's simplex & 38.7677  & 7279 &   1.4060e+03 \\
&&&& Dual simplex & + & + & +       \\ \hline
\caption{Test of Algorithm \ref{mainAlg} on Netlib problems}
\label{netlibTable}
\end{longtable}
\normalsize

Dantzig's most negative simplex, dual simplex, and the facet pivot simplex (using maximal deviation and maximal normalized deviation rules)
methods are implemented in Matlab to solve this set of Netlib benchmark problems
\footnote{\begin{samepage}We do not have the access to commercial software and have not
implemented the hiding tricks available inside commercial software,
therefore we implemented only the basic functions given in the
literature. But this is the same for all four implemented algorithms. 
We expect that the comparison should be fair enough.\end{samepage}}.
The test results are summarized in Table \ref{netlibTable}.
The dual simplex algorithm can be more efficient than other algorithms for some testing problems (for example, fit1d and fit2d discussed in the next subsection) as noted in 
\cite{koberstein05,ks07,kostina02,maros03} 
but can fail on rank deficient problems (they are marked by '*') or when an intermediate base is ill conditioned that causes the dual simplex algorithm erroneously claims the problems are 
dual unbounded (they are marked by '+' in this case)\footnote{Since Netlib benchmark set selects most difficult problems including rank deficiency ones, the performance of dual simplex method in Table~\ref{netlibTable} is not a representative of the performance for real-world problems.}, while Dantzig's simplex 
and facet simplex methods are much robust for these tested problems.  
Dantzig's simplex method usually takes more CPU times than
dual simplex and facet simplex methods (for algorithms fail to 
find a solution after $100000$ iterations or more than 20 hours, 
they are marked by '-' ).
As it can be seen in Table \ref{netlibTable} 
the facet pivot simplex algorithm performs very well in 
solving these problems because it uses less CPU time than 
Dantzig's most negative pivot rule algorithm for all problems except
{\tt grow15, grow7} for which Dantzig’s algorithm uses less
CPU time. Moreover, the facet pivot simplex algorithm solves
all these problems successfully.

Performance profile
\footnote{To our best knowledge, performance 
profile was first used in \cite{ty96} to compare the performance
of different algorithms. The method becomes very 
popular after its merit was carefully analyzed in \cite{dm02}.}
is used to compare the efficiency of three algorithms
(the dual simplex method is not included
because it fails to solve many tested problems). Figure~1 
is the performance profile of iteration numbers and Figure 2
is the performance profile of CPU times used by the three algorithms.
It clearly shows that facet pivot algorithm is more efficient than
Dantzig's most negative algorithm in general.

\begin{figure}[htb]
\centerline{\includegraphics[height=5cm,width=10cm]
{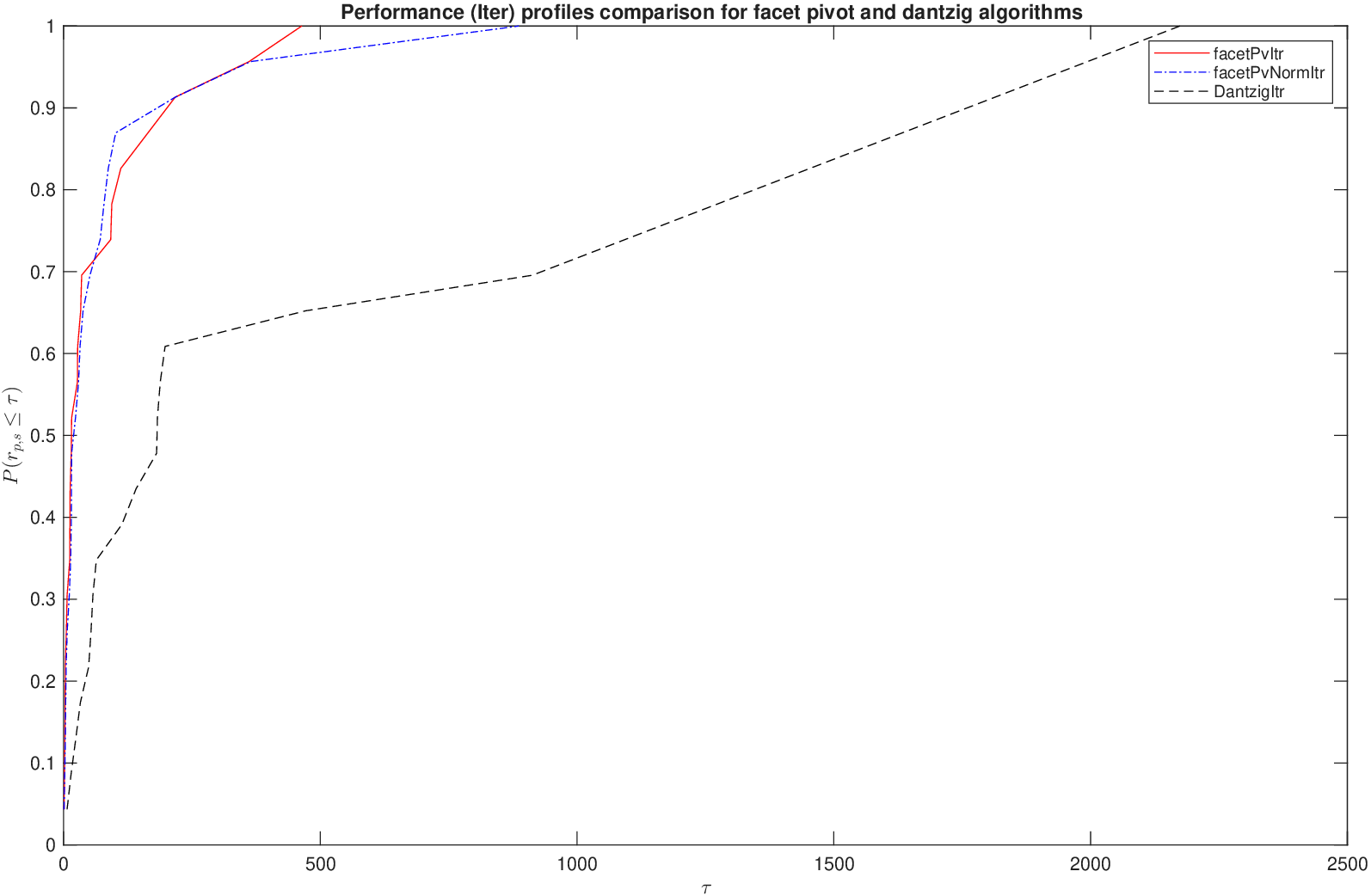}}
\caption{Performace (iteration number) profiles of the facet pivot and Dantzig's algorithms}
\label{case1}
\end{figure}

\begin{figure}[htb]
\centerline{\includegraphics[height=5cm,width=10cm]
{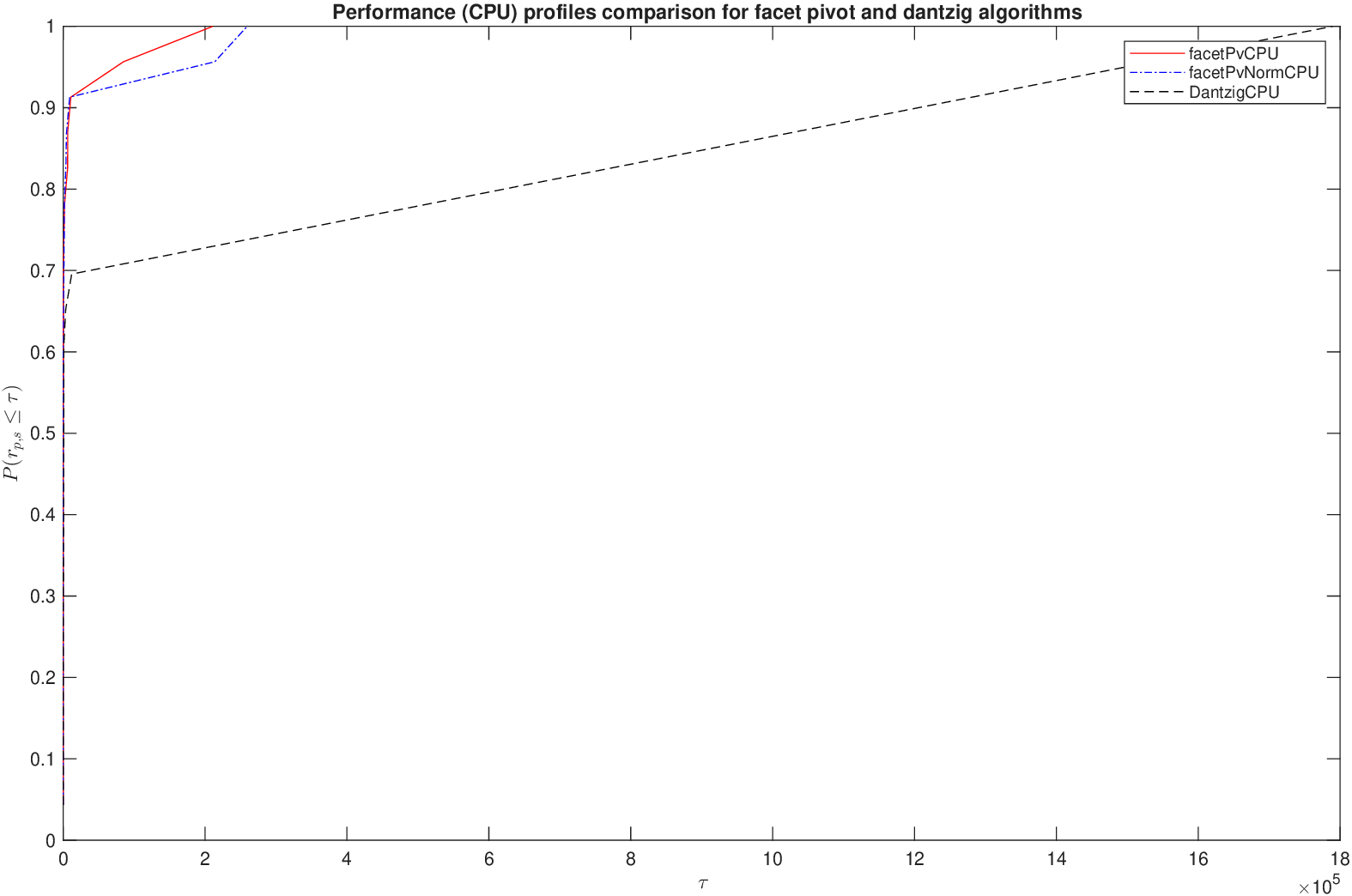}}
\caption{Performace (CPU) profiles of the facet pivot and Dantzig's algorithms}
\label{case2}
\end{figure}

\subsection{Test on Netlib benchmark problems where $d/m$ is large}\label{testSection1}

A reviewer suggested testing problems with $m<100$ and $d>100000$. It is known that some algorithms perform poorly for this type of problem. We examined all Netlib problems and found no problem satisfying these conditions. The closest ones are: FIT2D (m= 24, d=1049) and FIT2B (m= 26, d=10500). Following the suggestion of this reviewer,
we tested these problems using Dantzig's simplex method, the proposed facet pivot method, dual simplex method, plus two interior-point methods (IPM): Mehrotra's predictor-corrector method, and a recently developed arc-search IPM curveLP proposed in \cite{yang17}. All algorithms are implemented in Matlab, facet pivot code is available in Matlab file exchange site
\newline\footnotesize https://www.mathworks.com/matlabcentral/fileexchange/181706-facet-pivot-algorithm-for-linear-programming .
\newline\normalsize curveLP and Mehrotra's predictor-corrector codes are available in \cite{yang17} and
\newline https://www.mathworks.com/matlabcentral/fileexchange/53911-curvelp .
\newline 
Dantzig's most negative method is implemented in \cite{yang20a}.
The test results are summarized in the following table.
\footnotesize
\begin{longtable}{|c|c|c|c|c|c|c|c|}
	\hline          
	{Problem name} &  $m$ & $n$ & $d$ &  method  &  CPU  &   iter & obj 
	\\ \hline
	fit1d & 24 & 0 & 1049   &    Alg. \ref{mainAlg} 
	& 6.3799 & 628 & -9.1464e+03  \\
	&&&& Dantzig's simplex  & 87.7639   & 9454 &  -9.1464e+03 \\ 
	&&&& Dual simplex &  0.6051 & 2032 &  -4.6852e+03 \\ 
	&&&& Mehrotra's method & 1.195457  & 22 & -9.1464e+03 \\ 
	&&&& curveLP & 1.888994  & 23  & -9.1464e+03 \\ \hline
	fit2d  & 25  & 0 &  10524   &    Alg. \ref{mainAlg} 
	& 5.1125e+03  & 9253 & -6.8464e+04   \\
	&&&& Dantzig's simplex & 8.3172e+03  & 87594 &  -6.8464e+04      \\
	&&&& Dual simplex & 174.5856  & 25095 & -2.2948e+04 \\ 
	&&&& Mehrotra's method & 212.653359  & 26 & -6.8464e+04 \\ 
	&&&& curveLP & 193.126972  & 26  & -6.8464e+04 \\ \hline
	\caption{Test on Netlib problems with large ratio of $d/m$.}
	\label{ratiod2m}
\end{longtable}
\normalsize

For this type of problem, dual simplex method is the most efficient among all tested algorithms, followed by interior-point methods.

\subsection{Test on large Netlib benchmark problems with $n>50000$}

The reviewers are also suggested to test large Netlib benchmark problems with $n>50000$. We examined all Netlib problems and listed all problem in this category: CRE-B (m=9649  n=72447), CRE-D (m= 4351 n= 69980), KEN-18 (m=105128 n=154699), OSA-14 (m=2338   n=52460), OSA-30  (m= 4351 n=100024), OSA-60   (m=10281  n=232966), PDS-20 (m= 33875 n=105728). It is widely believed that interior-point method (IPM) is competitive for large scale problems. 

Following the suggestion of the reviewers, two efficient IPM algorithms (Mehrotra's predictor-corrector and arc-search) implemented in \cite{yang17} are tested against Dantzig's vertex pivot, facet pivot, and dual simplex methods. 
The results are listed in Tabel \ref{largeProblems}.
\footnotesize
\begin{longtable}{|c|c|c|c|c|c|c|c|}
	\hline          
	{Problem name} &  $m$ & $n$ & $d$ &  method  &  CPU  &   iter & obj 
	\\ \hline
	osa\_14  & 2338  & 0  & 52460  &  maximal deviation &  2.1178e+04    &  3669 & 1.1065e+06     \\
	&&&& maximal normalized deviation & 1.3508e+04  & 2648 &  1.1065e+06         \\
	&&&& Dantzig's simplex & -  & 100000 &  -      \\
	&&&& Dual simplex & +  & + & +     \\
	&&&& Mehrotra's predictor-corrector & 1.15727e+3  & 37 &  1.1065e+06   \\ 
	&&&& curveLP IPM & 1.02126e+3  & 33  &  1.1065e+06   \\ \hline

	osa\_30  & 4351  & 0  & 100024  &  maximal deviation &  -    &  - & -     \\
	&&&& maximal normalized deviation & -  & - &  -         \\
	&&&& Dantzig's simplex  & -  & 100000 &  -      \\
	&&&& Dual simplex & +  & + & +    \\
	&&&& Mehrotra's method & 4282.715520  & 36 & 2.1421e+06 \\ 
	&&&& arc-search IPM & 4300.476209  & 47  & 2.1421e+06 \\ \hline

	cre\_d  & 4351  & 0 &  69980   &   maximal deviation &  1.2322e+05    &  28535 & 2.4455e+07     \\
	&&&& maximal normalized deviation & 1.1237e+05  & 19328 &  2.4455e+07         \\
	&&&& Dantzig's simplex &-  & 100000 &  -      \\
	&&&& Dual simplex & +  & + & +  \\ 
	&&&& Mehrotra's predictor-corrector & + & + & + \\ 
	&&&& curveLP IPM & 292.682423  & 45  & 2.4455e+07 \\ \hline

	cre\_b  & 9649  & 0 &  72447   &   maximal deviation &  3.54835e+05   &  41078 & 2.3130e+07     \\
	&&&& maximal normalized deviation & 3.45125e+05  & 32438 &  2.3130e+07         \\
	&&&& Dantzig's simplex &-  & 100000 &  -      \\
	&&&& Dual simplex & +  & + & +  \\ 
	&&&& Mehrotra's predictor-corrector & 485.638819  & 49 &  2.3130e+07  \\ 
	&&&& curveLP IPM & 477.512233  & 47  & 2.3130e+07 \\ \hline

	\caption{Test on Netlib problems with large ratio of $d/m$.}
	\label{largeProblems}
\end{longtable}
\normalsize
In this test, for any problem, if an algorithm does not find an optimal solution in $96$ hours, it is marked as ``-''.
Dantzig's algorithm cannot find solutions in $100000$ iterations or  in $96$ hours for any of these problems. The dual simplex algorithm reported that these problems are dual unbounded because of rank deficiency or ill condition. Starting on OSA-30 ($n>100000$), facet pivot algorithm cannot find the solution in $96$ hours, therefore, the test stops at this point. Interior-point algorithms clearly perform much better for these large problems, they are significantly faster than other methods. For CRE-D, Mehrotra's predictor-corrector (MPC) failed because ``Matrix is close to singular or badly scaled''. Arc-search IPM performs best for this class of problems. Given this fact, facet pivot method is still important because Smale’s 9th problem asks researchers to find pivot methods that solve linear programming problem in polynomial bounds, and new pivot methods provide more ways to tackle the problem.
	
	
\subsection{Test on small size cycling problems}

A set of small size cycling problems is collected in \cite{yang21}. 
Algorithm \ref{mainAlg} has successfully solved all 30 problems
in this benchmark test set. Cycling does not happen for this set
of testing problems. However, this does not mean that the facet
pivot simplex algorithm using the maximal 
deviation rule will prevent the cycling problem from happening. 

\subsection{Test on Klee-Minty cube problems}

Klee-Minty cube and its variants have been used to prove that
several popular (vertex pivot) simplex algorithms need an exponential 
number (related to the problem size) of iterations in the worst case 
to find the optimal solution 
 \cite{ac78,friedmann11,gs79,jeroslow73}. 
In this section, three variants of Klee-Minty cube 
\cite{greenberg97,km11,ibrahima13} are used to test 
the facet pivot simplex algorithm.

The first variant of Klee-Minty cube is given in \cite{greenberg97}:
\begin{eqnarray}
\begin{array}{cl}
\min & -\sum_{i=1}^d 2^{d-i} x_i \\
\mbox{subject to} & 
\left[ 
\begin{array}{cccccc}
1 & 0 & 0 & \ldots & 0 & 0 \\
4 & 1 & 0 & \ldots & 0 & 0 \\
8 &4  & 1 & \ldots & 0 & 0 \\
\vdots &  \vdots & \vdots &  \ddots & 0 & 0 \\
2^{d-1} &  2^{d-2} & 2^{d-3} & \ldots  & 1 & 0 \\
2^d &  2^{d-1} & 2^{d-2} & \ldots & 4  & 1 \\
\end{array}
\right]
\left[ \begin{array}{c}
x_1 \\ x_2 \\ \vdots \\  \vdots \\ x_{d-1} \\ x_d
\end{array} \right] 
\le 
\left[ \begin{array}{c}
5 \\ 25 \\ \vdots \\  \vdots \\ 5^{d-1} \\ 5^d
\end{array} \right] 
  \\
& x_i \ge 0 \hspace{0.1in} i=1, \ldots, d.
\end{array}
\label{1stProblem}
\end{eqnarray}
The optimizer is $[ 0, \ldots, 0,5^d ]$ with optimal 
objective function $-5^d$.

The second variant of Klee-Minty cube is given in \cite{km11}:
\begin{eqnarray}
\begin{array}{cl}
\min & -\sum_{i=1}^d   x_i \\
\mbox{subject to} & 
x_1 \le 1, \\
& 2 \sum_{i=1}^{k-1} x_i +x_k \le 2^k-1
 \hspace{0.1in} k=2, \ldots, d, 
\\
& x_i \ge 0 \hspace{0.1in} i=1, \ldots, d.
  \\
\iff \min & [-1, -1, \ldots, -1, -1] \x 
\\
& \left[ 
\begin{array}{cccccc}
1 & 0 & 0 & \ldots & 0 & 0 \\
2 & 1 & 0 & \ldots & 0 & 0 \\
2 & 2  & 1 & \ldots & 0 & 0 \\
\vdots &  \vdots & \vdots &  \ddots & 0 & 0 \\
2 &  2 & 2 & \ldots  & 1 & 0 \\
2 &  2  & 2  & \ldots &  2   & 1 \\
\end{array}
\right]
\left[ \begin{array}{c}
x_1 \\ x_2 \\ \vdots \\  \vdots \\ x_{d-1} \\ x_d
\end{array} \right] 
\le 
\left[ \begin{array}{c}
1 \\ 3 \\ \vdots \\  \vdots \\ 2^{d-1}-1 \\ 2^d-1
\end{array} \right]  \\
& x_i \ge 0 \hspace{0.1in} i=1, \ldots, d.
\end{array}
\label{3rdProblem}
\end{eqnarray}
The optimizer is $[ 0, \ldots, 0, 2^d-1]$ with optimal 
objective function $2^d-1)$.

The third variant of Klee-Minty cube is given in \cite{ibrahima13}:
\begin{eqnarray}
\begin{array}{cl}
\min & -\sum_{i=1}^d 10^{d-i} x_i \\
\mbox{subject to} & 
2 \sum_{j=1}^{i-1} 10^{i-j} x_j +x_i \le 100^{i-1}
 \hspace{0.1in} i=1, \ldots, d,  
\\
& x_i \ge 0 \hspace{0.1in} i=1, \ldots, d.
  \\
\iff \min & [-10^{d-1}, -10^{d-2}, \ldots, -10, -1] \x 
\\
& \left[ 
\begin{array}{cccccc}
1 & 0 & 0 & \ldots & 0 & 0 \\
20 & 1 & 0 & \ldots & 0 & 0 \\
200 & 20  & 1 & \ldots & 0 & 0 \\
\vdots &  \vdots & \vdots &  \ddots & 0 & 0 \\
2(10^{d-2}) &  2(10^{d-3}) & 2(10^{d-4}) & \ldots  & 1 & 0 \\
2(10^{d-1}) &  2(10^{d-2}) & 2(10^{d-3}) & \ldots &  20  & 1 \\
\end{array}
\right]
\left[ \begin{array}{c}
x_1 \\ x_2 \\ \vdots \\  \vdots \\ x_{d-1} \\ x_d
\end{array} \right] 
\le 
\left[ \begin{array}{c}
1 \\ 100 \\ \vdots \\  \vdots \\ 10^{2(d-2)} \\ 10^{2(d-1)}
\end{array} \right]  \\
& x_i \ge 0 \hspace{0.1in} i=1, \ldots, d.
\end{array}
\label{2ndProblem}
\end{eqnarray}
The optimizer is $[ 0, \ldots, 0,10^{2(d-1)}]$ with optimal 
objective function $-10^{2(d-1)}$.


\begin{longtable}{|c|c|c|c|c|c|c|}
\hline          
\multirow{2}{*}{Problem} & \multicolumn{2}{|c|}
{Klee-Minty Variant 1 \cite{greenberg97}}  
& \multicolumn{2}{|c|} {Klee-Minty Variant 2 \cite{km11}}  & \multicolumn{2}{|c|} {Klee-Minty Variant 3 \cite{ibrahima13}}  \\ \cline{2-7}
{size}  & {Dantzig's rule}  & Alg. \ref{mainAlg} & {Dantzig' rule} & Alg. \ref{mainAlg}   & {Dantzig' rule} & Alg. \ref{mainAlg}    
\\ 
\hline
3     & 7                &   3  & $2^3-1$   & 3   & $2^3-1$   & 3    \\ \hline
4     &  15             &  4   & $2^4-1$    & 4   & $2^4-1$   & 4    \\ \hline
5     &  31             &  5   & $2^5-1$   & 5    & $2^5-1$   & 5    \\ \hline
6     &  63             &  6   & $2^6-1$    & 6     & $2^6-1$    & 6    \\ \hline
7     &  127           &  7    & $2^7-1$   & 7     & $2^7-1$   & 7     \\ \hline   
8     &  255           &  8   & $2^8-1$    & 8    & $2^8-1$    & 8       \\ \hline
9     &  511           &  9   & $2^9-1$    & 9     & $2^9-1$    & 9   \\ \hline
10   & 1023           &  10 & $2^{10}-1$    & 10  & $2^{10}-1$    & -    \\ \hline
11 & $2^{11}-1$    &  11 & $2^{11}-1$  & 11    & $2^{11}-1$  & -     \\ \hline
12 & $2^{12}-1$    &  12 & $2^{12}-1$  & 12    & $2^{12}-1$  & -      \\ \hline
13 & $2^{13}-1$    &  13 & $2^{13}-1$   &  13    & $2^{13}-1$   &  -  \\ \hline
14 & $2^{14}-1$    &  14 & $2^{14}-1$   & 14    & $2^{14}-1$   & -      \\ \hline
15 & $2^{15}-1$    &  15 & $2^{15}-1$   &  15     & $2^{15}-1$   &  -      \\ \hline
16 & $2^{16}-1$    &  16 & $2^{16}-1$   & 16  &   $2^{16}-1$   & -     \\ \hline
17   &  $2^{17}-1$                &  17  & $2^{17}-1$  & 17   &  $2^{17}-1$  & -     \\ \hline
18   &   $2^{18}-1$   &  18 & $2^{18}-1$   & 18   &  $2^{18}-1$   & -    \\ \hline
19   &  $2^{19}-1$   & 19  & $2^{19}-1$   &  19  &    $2^{19}-1$   &  -     \\ \hline
\caption{Iteration count comparison for Dantzig pivot and facet pivot
algorithms for the two Klee-Minty variants}
\label{table1}
\end{longtable}

The proposed facet pivot simplex algorithm is much more
efficient than Dantzig's most negative simplex method
for the first two Klee-Minty cube variants. For the third Klee-Minty cube variant, the facet pivot algorithm failed for problem size greater than $10$ because ``Matrix is close to singular or badly scaled'' (matrix $A$ has a poor condition number). The iteration
counts for Dantzig's most negative simplex method and facet
pivot simplex algorithm are listed in Table \ref{table1}.

\subsection{Test on randomly generated problems} \label{randomTest}

The facet pivot simplex algorithm has been tested for randomly generated
problems which are obtained as follows: first, given the problem
size $m$, an $m \times m$ matrix $\M$ with uniformly distributed
random entries between $[-0.5, 0.5]$ and an $m$ dimensional 
identity matrix are generated, then $\A=[\M~~~ \I]$ is 
determined and the initial base is composed of the last 
$m$ columns; second, a positive $m$-dimensional vector $\b$ 
whose entries are uniformly distributed between $[10, 11]$ is generated; 
third, a $m$-dimensional vector $\c_1$ whose entries are uniformly 
distributed between $[-0.5, 0.5]$ is generated, and $\c$ is 
given as $\c=(\c_1, \0)$. Therefore, linear programming problem
can be written as 
\begin{eqnarray}
\min & \c^{\T} \x \nonumber \\
s.t. & \A \x = \b \nonumber \\
 & \x \geq \0.
\label{simpleLP}
\end{eqnarray}
Clearly, Dantzig's pivot algorithm does not need a Phase I for this type of problem.
For each of these randomly generated 
standard LP problems, the Matlab codes for Dantzig's pivot algorithm 
and the facet pivot simplex algorithm are used to find the optimal solution.
For each given problem size $m$, this test is repeated for 
$100$ randomly generated problems. The average iteration
number and average computational time in seconds are 
obtained. The test results are presented in Table \ref{randomTable1}.
It is easy to see that for all problems with different sizes, 
the Dantzig algorithm uses less CPU time 
on average than the facet pivot simplex algorithm does. 

\begin{longtable}{|c|c|c|c|c|}
\hline          
\multirow{2}{*}{Problem} & \multicolumn{2}{|c|}
{Dantzig pivot algorithm}  
& \multicolumn{2}{|c|}{facet pivot simplex algorithm}   
\\ \cline{2-5}
{size m}  & iteration & CPU time (s) & iteration & CPU time (s)  \\ 
\hline
10      & 6.4100  & 0.0004   &   6.7900   &  0.0029    \\ \hline
100    & 160.8600  & 0.0778   &   161.3100   &  0.2558    \\ \hline
1000    & 1.9661e+4  & 1.0998e+3   &   1.9663e+4  &  5.1596e+3    \\ \hline
\caption{Comparison of the test results of Dantzig pivot and row 
pivot algorithms for random standard problems}
\label{randomTable1}
\end{longtable}

Although this test indicates that Dantzig's algorithm is more efficient 
than facet pivot simplex algorithm for standard LP and canonical LP 
problems if ``phase 1'' is not required for these problems, 
the next test shows that the facet pivot simplex algorithm is more 
efficient than Dantzig's algorithm for general LP problems
for which ``phase 1'' is indeed required. 
For the general LP problem (\ref{generalLP}), we can apply 
Dantzig's algorithm if we convert (\ref{generalLP}) as:
\begin{subequations}
\begin{align}
\min \hspace{0.5in}  &  [\c^{\Tr}~\0^{\T}~\0^{\T}~\0^{\T}]
(\x,~\y ,~\z,~\w)  \\ 
\mbox{\rm subject to} \hspace{0.3in}
&  \left[ \begin{array}{ccccc} 
\A^I & \0  & \0  & \0   \\ 
\A^J & -\I  & \0  & \0   \\
\I  & \0 & \I  & \0    \\
\I  & \0  & \0   &   -\I   
\end{array} \right]
 \left[ \begin{array}{c}
\x \\ \y \\ \z \\ \w \end{array} \right]
= \left[ \begin{array}{c} 
\b_I \\  \b_J \\ \u \\ {\Bell}
\end{array} \right]  \\
& ( \x, \y, \z, \w) \ge \0. 
\end{align}
\label{convertedLP}
\end{subequations}
Clearly, comparing to the general LP problem (\ref{generalLP}),
the standard general LP problem (\ref{convertedLP}) is ``bigger'' 
and it does need a ``phase 1'' to find the initial basic feasible solution 
if Dantzig's algorithm is used to solve it. Note that the facet pivot 
simplex algorithm needs only to solve the smaller general 
LP problem (\ref{generalLP}) and ``phase 1'' is not needed. 
The test result is given in Table \ref{randomTable2} and it 
clearly shows that facet pivot simplex algorithm is more efficient than 
Dantzig's algorithm for the general LP problem.

\begin{longtable}{|c|c|c|c|c|c|c|}
\hline          
\multicolumn{3}{|c|}{Problem size} & \multicolumn{2}{|c|}
{Dantzig pivot simplex algorithm}  
& \multicolumn{2}{|c|}{facet pivot simplex algorithm}  \\ \hline
$m$ & $n$ & $m_i$  & iteration & CPU time (s) 
& iteration & CPU time (s)  \\ \hline
1 & 5 & 3   & 13.7058  & 0.0150   &   14.6450   &  0.0115    \\ \hline
5 & 30 & 10    & 123.8073  & 0.0851   &   123.7481  &  0.0181    \\ \hline
10 & 100 & 20    & 562.2700  & 1.7319   &   563.2700  &  0.1190    \\ \hline
50 & 500 & 100  & 8.7134e+03  & 697.0669   &   8.7143e+03  &  21.7058    \\ \hline
\caption{Comparison of the test results of Dantzig pivot and facet pivot 
algorithms for correlated random general problems}
\label{randomTable2}
\end{longtable}

Finally, we generated test problems (\ref{convertedLP}) by using Gaussian distributed data and solved the problems by Dantzig pivot simplex algorithm and facet pivot simplex algorithm as suggested by a reviewer. We generated 100 instances for each test case. The test results are provided in Table \ref{randomTable3}.

\begin{longtable}{|c|c|c|c|c|c|c|}
	\hline          
	\multicolumn{3}{|c|}{Problem size} & \multicolumn{2}{|c|}
	{Dantzig pivot simplex algorithm}  
	& \multicolumn{2}{|c|}{facet pivot simplex algorithm}  \\ \hline
	$m$ & $n$ & $m_i$  & iteration & CPU time (s) 
	& iteration & CPU time (s)  \\ \hline
	1 & 5 & 3   & 16.8333  & 0.0032   &   17.5556   &  0.0022   \\ \hline
	5 & 30 & 10    & 127.7000  & 0.0180   &   128.7000  &  0.0054    \\ \hline
	10 & 100 & 20    & 495.8500  & 0.6712   &   496.8500  &  0.0148    \\ \hline
	50 & 500 & 100  & 4.5541e+03  & 214.9355   &   4.5551e+3  &  5.5354    \\ \hline
	\caption{Comparison of the test results of Dantzig pivot and facet pivot 
		algorithms for Gaussian distributed random general problems}
	\label{randomTable3}
\end{longtable}
It is clear that for problems with the format (\ref{convertedLP}) (which does need a phase I for Dantzig's simplex method), facet pivot algorithm is significantly more efficient.

\section{Conclusion}

In this paper, we proposed a facet pivot simplex algorithm. 
It is proven that the facet pivot simplex algorithm 
finds the optimal solution in finite iterations if the least index 
rule is used in the selection of entering/leaving rows (facets).
Two Matlab functions are developed to implement the 
facet pivot simplex algorithm. Numerical test is performed.
The test result shows that
the facet pivot simplex algorithm is more efficient than Dantzig's 
most negative pivot rule algorithm and is more robust than
dual simplex method for Netlib LP problems when Phase I is required. 
Additionally, the facet pivot simplex algorithm is more 
efficient than Dantzig's most negative pivot rule algorithm for 
some specially designed hard problems, such as cycling LP 
problems and Klee-Minty problems. For large scale problems, as expected, interior-point method is more efficient than pivot method.

\section{Acknowledgment}

The author would like to thank Dr. Y. Liu for sharing his excellent 
paper which is very helpful in the preparation of this work. Dr. Robert E. Pritchett of Goddard Space Flight Center at NASA reviewed an earlier version of the paper and helped the author to improve the presentation of the paper.

%

%
%
%
%


\end{document}